\documentclass[11pt,final]{article}

\usepackage{amsmath, amsfonts}
\numberwithin{equation}{section}
\numberwithin{table}{section}
\numberwithin{figure}{section}
\usepackage{bbm}
\usepackage[headings]{fullpage}
\usepackage{graphicx}
\usepackage{caption}
\usepackage{enumerate}
\usepackage{multirow}
\usepackage{float}
\usepackage[
            colorlinks=true,%
            breaklinks=true,%
            linkcolor=blue,
            urlcolor=blue,%
            citecolor=blue,%
            pdfauthor={D. Grundvig and M. Heinkenschloss},%
            bookmarksopen=false,%
            pdfpagemode=None]{hyperref}
\usepackage{times}
\usepackage{tikz}
\usepackage{color}
\usepackage{mathrsfs}
\usepackage[ruled,vlined]{algorithm2e}
\usepackage{algorithmic}

\newcommand{\bpm}{\begin{pmatrix}}
\newcommand{\epm}{\end{pmatrix}}

\newcount\colveccount
\newcommand*\colvec[1]{
        \global\colveccount#1
        \begin{bmatrix}
        \colvecnext
}
\def\colvecnext#1{
        #1
        \global\advance\colveccount-1
        \ifnum\colveccount>0
                \\
                \expandafter\colvecnext
        \else
                \end{bmatrix}
        \fi
}

\newcommand {\real} {\mathbb{R}}

\newcommand {\cB} {\mathcal{B}}

\newcommand {\cF} {\mathcal{F}}

\newcommand{\BG}{\ensuremath{\mathbf{G}} } %
\newcommand{\BI}{\ensuremath{\mathbf{I}} } %
\newcommand{\BK}{\ensuremath{\mathbf{K}} } %
\newcommand{\BQ}{\ensuremath{\mathbf{Q}} } %
\newcommand{\BR}{\ensuremath{\mathbf{R}} } %
\newcommand{\BT}{\ensuremath{\mathbf{T}} } %
\newcommand{\BV}{\ensuremath{\mathbf{V}} } %
\newcommand{\bd}{\ensuremath{\mathbf{d}}} %
\newcommand{\bg}{\ensuremath{\mathbf{g}}} %
\newcommand{\bp}{\ensuremath{\mathbf{p}}} %
\newcommand{\bu}{\ensuremath{\mathbf{u}}} %
\newcommand{\bv}{\ensuremath{\mathbf{v}}} %
\newcommand{\by}{\ensuremath{\mathbf{y}}} %
\newcommand{\bzero}{\ensuremath{\mathbf{0}}} %

\newcommand {\blambda} {\mbox{\boldmath $\lambda$}}

\newtheorem{theorem}{Theorem}[section]
\newtheorem{lemma}[theorem]{Lemma}
\newtheorem{corollary}[theorem]{Corollary}

\newtheorem{assumptions}[theorem]{Assumptions}

\newenvironment{proof}%
               {\noindent {\bf Proof:} }%
               {\hfill $\Box$ \\[1ex] }

               {\refstepcounter{theorem}\vspace{2ex}%
                \noindent{\bf Examples \thetheorem} }%
               {\hfill $\diamond$}

               {\refstepcounter{theorem}%
                \noindent{\bf Example \thetheorem} }%
               {\hfill $\diamond$}

\newcounter{problem} 
\renewcommand{\theproblem} {\arabic{problem}}
               {\refstepcounter{problem} \vspace{2ex}%
                \noindent{\bf Problem \theproblem} }%
               { }

               {\noindent{\bf Solution}}%
               {\hfill $\Box$ \\[1ex]}

\newlength{\boxwidth}

\newlength{\fullboxwidth}
\newlength{\fullinboxwidth}
\SetKwComment{TComment}{$\triangleright\;$}{}

\begin{document}
\title{{\Large \bf A Generalized $\ell_1$-Merit Function SQP Method Using Function Approximations with Tunable Accuracy}
         \thanks{This research was supported in part by AFOSR Grant FA9550-22-1-0004, NSF Grant DMS-2231482,  and a 
                      2022 National Defense Science and Engineering Graduate (NDSEG) Fellowship }
         }
\author{Dane S.\ Grundvig
             \thanks{Department of Computational Applied Mathematics and Operations Research,
                      MS-134, Rice University, 6100 Main Street,
                     Houston, TX 77005-1892. 
                     E-mail: dsg12@rice.edu}
             \and
             Matthias Heinkenschloss
             \thanks{Department of Computational Applied Mathematics  and Operations Research,
                 MS-134, Rice University, 6100 Main Street,
                 Houston, TX 77005-1892, and the Ken Kennedy Institute, Rice University.
                 E-mail: heinken@rice.edu}
        }

\date{\today}

\maketitle

\begin{abstract}
This paper develops a generalization of the line-search sequential quadratic programming (SQP) algorithm with $\ell_1$-merit function
that uses objective and constraint function approximations with tunable accuracy to solve smooth equality-constrained optimization problems. 
The evaluation of objective and constraint functions and their gradients is potentially computationally expensive, but it is assumed that 
one can construct effective, computationally inexpensive models of these functions.   This paper specifies how these models can be 
used to generate new iterates.  
At each iteration, the models have to satisfy function error and relative gradient error tolerances determined by the algorithm based on
 its progress. Moreover, bounds for the model errors are used to explore regions where the combined objective function and constraint 
 models are sufficiently accurate. 
 The algorithm has the same first-order global convergence properties as a line-search SQP algorithm with $\ell_1$-merit function, 
 but only uses objective and constraint function models and the model error bounds. 
 The algorithm is applied to  a discretized boundary control problem in which the evaluation of the objective and constraint functions
 requires the solution of the  Boussinesq partial differential equation (PDE).
 The models are constructed from projection-based reduced-order models of the Boussinesq PDE.
 \end{abstract}


\section{Introduction}
We develop a generalization of the line-search sequential quadratic programming (SQP) algorithm with $\ell_1$-merit function
that uses objective function and constraint function models with tunable accuracy to solve smooth equality-constrained optimization problems.
Given smooth functions $f:\real^n\rightarrow\real$ and $c:\real^n\rightarrow\real^m$, $m < n$, 
we consider the solution of 
\begin{subequations} \label{eq:exact-nlp}
\begin{align}
  \min \quad &  f(x), \\
  \text{s.t.}\quad & c(x)=0.
\end{align}
\end{subequations}
We assume that  evaluations of $f$ and $c$ and their derivatives are computationally expensive, e.g., because
an evaluation of $f$ and $c$ at $x$ requires an expensive simulation, but that one can compute differentiable,
effective, computationally inexpensive to evaluate models $m_k$ and $h_k$ of the objective function around the
current iterate $x_k$.
Instead of a traditional line-search SQP method that uses a quadratic model of the Lagrangian
\begin{equation} \label{eq:exact-lagrangian}
  \mathcal{L}(x,\lambda) = f(x) - \lambda^Tc(x)
\end{equation}
associated with \eqref{eq:exact-nlp} and a linear model of the constraints to compute a step, and,
ultimately, a new iterate, our algorithm uses models $m_k:\real^n\rightarrow\real$ and $h_k:\real^n\rightarrow\real^m$,
of the objective and constraint function built around the current iterate $x_k$ to generate a new iterate by approximately 
solving subproblems of the type
\begin{subequations} \label{eq:exact-nlp-sub}
\begin{align}
  \min \quad &  m_k(x), \\
  \text{s.t.}\quad & h_k(x)=0.
\end{align}
\end{subequations}
Because $m_k$, $h_k$,  and their derivatives can be evaluated inexpensively, an approximate solution
of \eqref{eq:exact-nlp-sub} can be computed efficiently. If $m_k$ and $h_k$ are good approximations of $f$ and $c$,
respectively, a new iterate generated by approximately solving \eqref{eq:exact-nlp-sub} is expected to be a better
approximation of the solution of  \eqref{eq:exact-nlp} than the current iterate $x_k$ is. Unfortunately, the models
$m_k$, $h_k$ are not good models of $f$, $c$ globally, but only in certain regions of the optimization variable space.
In this paper, we specify what approximation properties are required of these models at the current iterate $x_k$,  
and how exactly these models can be used to generate a new iterate using a generalization of the line-search SQP method 
with $\ell_1$-merit function.
While once generated, models $m_k$, $h_k$, and their derivatives are computationally inexpensive to evaluate, the computation of
these models carries a computational cost. Therefore, it is desirable to use current models as much as possible. 
This paper extends our recent work \cite{DSGrundvig_MHeinkenschloss_2025a} on unconstrained problems to the constrained
case \eqref{eq:exact-nlp} and demonstrates the potential of the new algorithm on a boundary control problem governed by the
Boussinesq partial differential equation (PDE).

Our set-up is motivated by problems \eqref{eq:exact-nlp} with objective and constraint functions of the form
\begin{equation} \label{eq:exact-nlp-implicit-constraints}
    f(x) = \widehat f(y(x);x), \qquad
    c(x) = \widehat c(y(x);x),
\end{equation}
where $f:\real^N \times \real^n \rightarrow\real$ and $c: \real^N \times \real^n \rightarrow\real^m$, and where $y(x) \in \real^N$, $N \gg 1$ ,
solves a large-scale system $R(y(x);x)=0$, where $R:\real^N \times \real^n \rightarrow\real^N$ often represents a discretized
PDE and $y(x) \in \real^N$ represents the discretized PDE solution for given parameters or controls $x \in \real^n$.
The system $R(y(x);x)=0$ is also referred to as the full order model (FOM).
A projection-based reduced order model (ROM) generates $V, W \in \real^{N \times r}$, $r\ll N$, and approximates $y(x)$ by
$V \widehat{y}(x)$, where $ \widehat{y}(x) \in \real^r$ solves the smaller scale ROM  $W^T R(V \widehat{y}(x);x)=0$. For
many nonlinear systems, so-called hyperreduction methods are needed to make the solution of the resulting ROM truly computationally
inexpensive. For details on ROMs see, e.g.,
\cite{PBenner_ACohen_MOhlberger_KWillcox_2017a}, \cite{PBenner_SGrivet-Talocia_AQuarteroni_GRozza_WSchilders_LMSilveira_2021b},
\cite{JSHesthaven_GRozza_BStamm_2015a}, \cite{AQuarteroni_AManzoni_FNegri_2016a}.
A ROM implies models 
\begin{equation} \label{eq:exact-nlp-implicit-constraints-models}
    m_k(x) = \widehat f(V \widehat{y}(x);x), \qquad
    h_k(x) = \widehat c(V \widehat{y}(x);x)
\end{equation}
of the objective and constraint functions \eqref{eq:exact-nlp-implicit-constraints}. The index $k$ signals that the ROM,
i.e, the matrices $V, W \in \real^{N \times r}$, are built from FOM information at the iterate $x_k$ and, possibly, at previous 
iterates. The approximation properties of the models \eqref{eq:exact-nlp-implicit-constraints-models} and of their derivatives can be controlled
by the selection of $V, W \in \real^{N \times r}$. We use ROMs to compute function approximations in the boundary control problem governed by the
Boussinesq PDE.
However, we note that our generalization of the line-search SQP algorithm with $\ell_1$-merit function is agnostic about how the 
function approximations are computed. It specifies approximation properties that these function approximations have to satisfy.
ROM-generated approximations are one way to generate these models.

There are many approaches for smooth unconstrained optimization problems that allow objective function models with tunable accuracy,
including  \cite{NAlexandrov_JEDennis_RMLewis_VTorczon_1998},  \cite{RGCarter_1991a}, \cite{RGCarter_1993a}, 
\cite{DPKouri_MHeinkenschloss_DRidzal_BGvanBloemenWaanders_2014a},
\cite{DSGrundvig_MHeinkenschloss_2025a},
\cite{YYue_KMeerbergen_2013a}. 
These and related approaches have been used, e.g., in
\cite{TKeil_LMechelli_MOhlberger_FSchindler_SVolkwein_2021a}, 
\cite{EQian_MAGrepl_KVeroy_KWillcox_2017a},
\cite{TWen_MJZahr_2023a},
\cite{MYano_THuang_MJZahr_2021a},
\cite{MJZahr_CFarhat_2015a} to manage the construction of ROMs to accelerate PDE-constrained optimization problems.
Our proposed algorithm builds on \cite{DSGrundvig_MHeinkenschloss_2025a}, \cite{YYue_KMeerbergen_2013a} and it
is equal to the approach in \cite{DSGrundvig_MHeinkenschloss_2025a}  if there are no constraints.
Line-search SQP methods with $\ell_1$-merit function go back to the work \cite{SPHan_1977a}, and have 
been successfully used to solve smooth equality-constrained problems \eqref{eq:exact-nlp}.
See, e.g.,  \cite[Ch.~17]{JFBonnans_JCGilbert_CLemarechal_CASagastizabal_2006a},
\cite{PTBoggs_JWTolle_1995a},
\cite[Ch.~18]{JNocedal_SJWright_2006a}, 
\cite{PEGill_EWong_2012a}.
In comparison to the unconstrained case, there are relatively few approaches to solve problems like \eqref{eq:exact-nlp}
using function models with tunable accuracy. Approaches include
\cite{WHess_SUlbrich_2013a},
\cite{TWen_MJZahr_2025a},
\cite{JCZiems_SUlbrich_2011a}, 
\cite{JCZiems_2013a}.
The last two papers consider models derived from PDE discretizations and include the PDE as
an explicit constraint. This is different from the setting we tackle, where a discretized PDE appears implicitly, if at all.
The paper \cite{WHess_SUlbrich_2013a} also considers a line-search SQP method with $\ell_1$-merit function,
for PDE-constrained optimization,  but in their problem, the constraints do not depend on the PDE solution and therefore
inexactness only appears in the objective function and its gradient, and inexactness is due to adaptive mesh refinement; no
ROMs are used. The paper \cite{TWen_MJZahr_2025a} uses an augmented Lagrangian approach and function 
approximations computed using ROMs with hyperreduction. The emphasis of \cite{TWen_MJZahr_2025a} is on
the integration of hyperreduced ROM into the solution of unconstrained optimization subproblems, but does not
analyze its impact in the outer augmented Lagrangian iteration. In contrast, we provide a convergence analysis of the 
complete proposed optimization framework.
We note that our problem setting is different from that in 
\cite{ASBerahas_FECurtis_DRobinson_BZhou_2021a}, \cite{FECurtis_DPRobinson_BZhou_2024a},
where stochastic methods are extended to equalty-constrained optimization.

This paper is organized as follows.
In the next Section~\ref{sec:exact-merit}, the classical  line-search SQP method 
with $\ell_1$-merit function is reviewed to inform our new algorithm development.
Section~\ref{sec:inexact-merit} develops the new algorithm using models $m_k$ and $h_k$
and subproblems related to \eqref{eq:exact-nlp-sub}. First, in Section~\ref{sec:inexact-merit-basic},
we will develop a basic algorithm using an idealized descent condition on which the convergence result is based. 
Then, in Section \ref{sec:inexact-algo-practical}, we will develop a practical algorithm that uses error bounds
for the models $m_k$ and $h_k$ to derive an implementable algorithm that enforces the idealized descent condition.
In Section~\ref{sec:numerics}, the new algorithm is applied to a boundary control problem governed by the
Boussinesq PDE, with models generated by ROMs.

\section{$\ell_1$-Merit Function SQP with Exact Function Information} \label{sec:exact-merit}
Our algorithm builds on the well-known SQP method with $\ell_1$-merit function.
See, for example \cite[Ch.~17]{JFBonnans_JCGilbert_CLemarechal_CASagastizabal_2006a},
\cite{PTBoggs_JWTolle_1995a},
\cite[Ch.~18]{JNocedal_SJWright_2006a}, 
\cite{PEGill_EWong_2012a}.
This section reviews the basic ingredients to provide background for the proposed
modifications necessary to allow the replacement of the objective and constraint functions by suitable models.

\subsection{Step Computation} \label{sec:exact-merit-step}
Let $c'(x)\in\real^{m\times n}$ denote the Jacobian of the constraints.
The first-order necessary conditions for  \eqref{eq:exact-nlp} are
\begin{equation} \label{eq:exact-fonc}
    \nabla f(x) - c'(x)^T\lambda =0 \quad \mbox{ and } \quad 
    c(x) =0.
\end{equation}
Given the iterate $x_k$, the search direction $s_k$ in the SQP method is chosen as the (approximate) solution of
\begin{subequations} \label{eq:s-prob}
  \begin{align} 
    \min \quad & \frac{1}{2} s^TH_ks + \nabla f(x_k)^T s, \label{eq:s-prob-a} \\
    \text{s.t.}\quad & c(x_k) + c'(x_k)s=0, \label{eq:s-prob-b}
  \end{align}
\end{subequations}
where $H_k\in \real^{n\times n}$ is a symmetric matrix that replaces the Hessian $\nabla^2_{xx} \mathcal{L}(x_k,\lambda_k)$.
The problem of finding the search direction can be seen as a quadratic approximation of the Lagrangian \eqref{eq:s-prob-a}
subject to a linearization of the equality constraint \eqref{eq:s-prob-b}.

Assume that  $c'(x_k) \in\real^{m\times n}$ has rank $m$.
If $H_k$ is symmetric positive semidefinite on the nullspace ${\cal N}(c'(x_k))$, then  $s_k \in \real^n$ solves \eqref{eq:s-prob} if and only if 
$s_k \in \real^n$, $\lambda_{k+1} \in \real^m$ solve  the following KKT system (see \cite[Sec. 18.2]{JNocedal_SJWright_2006a})
\begin{equation} \label{eq:exact-kkt}
  \begin{bmatrix} H_k & -c'(x_k)^T\\ c'(x_k) & 0\end{bmatrix}
  \begin{bmatrix} s_k \\ \lambda_{k+1}\end{bmatrix} 
 = -\begin{bmatrix} \nabla f(x_k) \\ c(x_k)\end{bmatrix}.
\end{equation}
Moreover, if $H_k$ is symmetric positive definite on the nullspace ${\cal N}(c'(x_k))$, then the KKT system \eqref{eq:exact-kkt}
has a unique solution and therefore, \eqref{eq:s-prob} has a unique solution.

One can also convert \eqref{eq:s-prob} to an unconstrained problem.
Again, assume that  $c'(x_k) \in\real^{m\times n}$ has rank $m$.
Let $s_k^c \in \real^n$ be a particular solution of \eqref{eq:s-prob-b}.
For example, $s_k^c$ could be set to the minimum norm solution of \eqref{eq:s-prob-b} given by
\begin{equation} \label{eq:s-prob-project-s0}
  s_k^c = - c'(x_k)^T (c'(x_k) c'(x_k)^T)^{-1} c(x_k).
\end{equation}
Let
\begin{equation} \label{eq:s-prob-project-P}
  P_k  = I - c'(x_k)^T (c'(x_k) c'(x_k)^T)^{-1} c'(x_k) \in \real^{n \times n}
\end{equation}
be the projection onto ${\cal N}(c'(x_k))$. If $m $ is small, \eqref{eq:s-prob-project-s0}
and \eqref{eq:s-prob-project-P} can be computed efficiently. 
Often the projection matrix \eqref{eq:s-prob-project-P} is not needed explicitly, but only
its application to vectors is required.

Because $s_k$ satisfies the linearized constraints \eqref{eq:s-prob-b} if and only if 
\begin{equation} \label{eq:s-prob-red-equivalence}
       s_k = s_k^c + P_k s_k^0, 
\end{equation}
the equality constrained quadratic program (QP) \eqref{eq:s-prob} is equivalent to the unconstrained problem
\begin{equation} \label{eq:s-prob-red}
    \min_{s^0 \in\real^n}\quad \frac{1}{2} (s^0)^T\big( P_k H_k P_k \big) s^0 + \Big( P_k \big( \nabla f(x_k) + H_k s_k^c \big) \Big)^T s^0.
\end{equation}
A vector $s_k \in \real^n$ solves the equality constrained QP \eqref{eq:s-prob} if and only
if there exists a solution $s_k^0 \in \real^n$ of \eqref{eq:s-prob-red} such that
\eqref{eq:s-prob-red-equivalence} holds.
The matrix $P_k H_k P_k$ is singular with nullspace  ${\cal N}(P_k H_k P_k) = {\cal N}(c'(x_k))^\perp = {\cal R}(c'(x_k)^T)$, but 
if $P_k H_k P_k$ is positive definite on ${\cal N}(c'(x_k))$, which is equivalent to $H_k$ being positive definite on ${\cal N}(c'(x_k))$, 
then \eqref{eq:s-prob-red} has a solution $s_k^0 \in \real^n$.
Because $P_k H_k P_k$ is singular with nullspace  ${\cal N}(P_k H_k P_k) = {\cal N}(c'(x_k))^\perp = {\cal R}(c'(x_k)^T)$,
a solution $s_k^0 \in \real^n$ of \eqref{eq:s-prob-red} is not unique, but $P_k  s_k^0$ is unique.
The unconstrained problem can be solved approximately, e.g., using the conjugate gradient method, which only uses
products of $P_k H_k P_k$ times a vector, but not $P_k H_k P_k$ explicitly.

If an approximate solution \eqref{eq:s-prob-red-equivalence} has been computed, the corresponding Lagrange multiplier
 $\lambda_{k+1} \in \real^m$ can be computed from the first equations in the KKT system \eqref{eq:exact-kkt}, i.e., 
 by solving $\min_{\lambda \in \real^m} \| c'(x_k)^T \lambda - \big( \nabla f(x_k) + H_k s_k \big) \|_2^2$.
 This gives
$\lambda_{k+1} =   (c'(x_k) c'(x_k)^T)^{-1} c'(x_k) \big( \nabla f(x_k) + H_k s_k \big)$.

\subsection{Step-Length Computation using $\ell_1$-Merit Function} \label{sec:exact-merit-length}

Given $s_k$, a solution to \eqref{eq:exact-kkt}, the proximal point $x_{k+1}$ will be chosen as
\begin{equation} \label{eq:x+1}
  x_{k+1} = x_k +\alpha_ks_k,
\end{equation}
where the step-size $\alpha_k$ is chosen to satisfy sufficient decrease in the $\ell_1$-merit function,
which is defined as
\begin{equation} \label{eq:exact-merit}
  \phi(x;\rho)= f(x) + \rho \| c(x)\|_1,
\end{equation}
where $\rho >0$ is the penalty parameter.

The sufficient decrease condition to be satisfied at each step is 
\begin{equation} \label{eq:exact-merit-sd}
  \phi(x_k +\alpha_ks_k;\rho_k)\leq \phi(x_k ;\rho_k) + c_1\alpha_k D(\phi(x_k ;\rho_k); s_k) 
\end{equation}
for $c_1\in(0,1)$ and $D(\phi(x_k ;\rho_k); s_k)$ the directional derivative of $\phi(x_k ;\rho_k)$ in the direction of $s_k$.
It can be shown (e.g., \cite[Proposition~17.1]{JFBonnans_JCGilbert_CLemarechal_CASagastizabal_2006a}
 or \cite[Thm.~18.2]{JNocedal_SJWright_2006a})
that the directional derivative of \eqref{eq:exact-merit} in a direction $s_k$ that satisfies $c'(x_k) s_k = -c(x_k)$
is
\begin{align} \label{eq:exact-dd}
  D(\phi(x_k ;\rho_k); s_k)  =  \nabla f(x_k)^Ts_k  - \rho_k \| c(x_k) \|_1.
\end{align}
From \eqref{eq:exact-kkt},
\[
  \nabla f(x_k)^Ts_k = -s_k^TH_ks_k + s_k^Tc'(x_k)^T\lambda_{k+1} =  -s_k^TH_ks_k - c(x_k)^T\lambda_{k+1} 
  \leq  -s_k^TH_ks_k + \|c(x_k)\|_1\|\lambda_{k+1}\|_\infty,
\]
which implies, along with \eqref{eq:exact-dd}, that 
\begin{equation} \label{eq:exact-dd-ineq}
  D(\phi(x_k ;\rho_k); s_k)\leq -s_k^TH_ks_k - (\rho_k -\|\lambda_{k+1}\|_\infty )\|c(x_k)\|_1.
\end{equation}

So, if $\rho_k>\|\lambda_{k+1}\|_\infty$ and $s_k^T H_k s_k \ge 0$,
then \eqref{eq:exact-dd-ineq} implies that $s_k$ is a descent direction for the merit function $\phi(\cdot; \rho_k)$ at $x_k$.
In this case, standard line-search methods, such as the backtracking line-search 
 \cite[p.~296]{JFBonnans_JCGilbert_CLemarechal_CASagastizabal_2006a},
\cite[Sec.~6.3.2]{JEDennis_RBSchnabel_1996}, \cite[Sec.~3.1]{JNocedal_SJWright_2006a}
can be used.  
Specifically, given $0<\beta_1\leq\beta_2<1$ and $\alpha_k^{(0)} > 0$,  we compute a steps size as follows.
If $\alpha_k^{(0)}$ satisfies  \eqref{eq:exact-merit-sd}, set $\alpha_k = \alpha_k^{(0)}$.
Otherwise, set  $i=1$.
Select $\alpha_k^{(i)}\in[\beta_1 \alpha_k^{(i-1)}, \beta_2 \alpha_k^{(i-1)}]$. 
If $\alpha_k^{(i)}$ satisfies  \eqref{eq:exact-merit-sd}, set $\alpha_k = \alpha_k^{(i)}$.
Otherwise replace $i$ by $i-1$ and repeat.

The condition $s_k^T H_k s_k \ge 0$ is guaranteed if $H_k$ is positive definite,  but weaker conditions on $H_k$ 
can be allowed.
The condition  $\rho_k>\|\lambda_{k+1}\|_\infty$ is enforced by adjusting the penalty parameter after computing the
step and the Lagrange multiplier $\lambda_{k+1}$ using
\begin{equation} \label{eq:rho}
          \rho_k= 
          \begin{cases} 
          \rho_{k-1}, \quad &\text{if } \rho_{k-1} \geq \|\lambda_{k+1}\|_\infty +\sigma, \\
           2\|\lambda_{k+1}\|_\infty +\sigma, \quad &\text{otherwise}, 
           \end{cases}
\end{equation}
where $\sigma > 0$ is a given constant.

The resulting line-search SQP algorithm with $\ell_1$-merit function is summarized in Algorithm~\ref{alg:exact}.

\begin{algorithm}[!htb]
	\caption{Line-Search SQP Algorithm} \label{alg:exact}
	\begin{algorithmic}[1]    
    \REQUIRE $c_1\in(0,1)$, $x_0$, $\rho_0$, $\sigma>0$, tolerances $\mbox{tol}_f,\mbox{tol}_c>0$
    \ENSURE Point $x_K$ where $\| \nabla f(x_K) - c'(x_K)^T\lambda_K\| < \mbox{tol}_f$, $\| c(x_K) \|_1 < \mbox{tol}_c$.
    \FOR{$k=0,1,2,\dots$}
            \STATE {\bf if} $\| \nabla f(x_k) - c'(x_k)^T\lambda_k\| < \mbox{tol}_f$ and $\| c(x_k) \|_1 < \mbox{tol}_c$ {\bf then} stop.
            \STATE Find $s_k$, $\lambda_{k+1}$ by solving \eqref{eq:exact-kkt}.
	    \STATE Compute $\rho_k$ using \eqref{eq:rho}.
	     \STATE Use backtracking line-search to find $\alpha_k$ that satisfies the sufficient decrease condition \eqref{eq:exact-merit-sd}.
              \STATE Set $x_{k+1}= x_k+\alpha_k s_k$.
     \ENDFOR
	\end{algorithmic}
\end{algorithm}

\subsection{Convergence} \label{sec:exact-merit-convergence}
We review the main ingredients of the convergence analysis. 

\begin{assumptions} \label{ass:exact}
	\begin{enumerate}[i)]
    \item The sequence $\{f(x_k)\}_{k=0}^\infty$ is bounded from below.
    \item The sequences $\{\|c(x_k)\|_1\}_{k=0}^\infty, \{\|c'(x_k)\|_1\}_{k=0}^\infty, \{\|\nabla f(x_k)\|_1\}_{k=0}^\infty $ are bounded.
    \item There exists $0 < h_l < h_u $ such that $h_l \| v \|_2 \le v^T H_k v  \le  h_u \| v \|_2$ for all $v \in \real^n$ and for all $k$. 
   \item The singular values of $c'(x_k)$ are uniformly bounded: $0 < \sigma_{\min} \leq \sigma_{\min}(c'(x_k))\leq \sigma_{\max}(c'(x_k)) \leq \sigma_{\max}$.
    \item The gradients $\nabla f$ and $\nabla c_i$, $i = 1, \ldots, m$,  are Lipschitz continuous with constants $L^f$ and
            $L^{c_i}$, $i = 1, \ldots, m$. The vector of Lipschitz  constants is $L^c=\big[ L^{c_1},  \ldots,  L^{c_m} \big]^T$.
	\end{enumerate}
\end{assumptions}
Note that  Assumption~\ref{ass:exact} means that $ \{\|H_k\|_2\}_{k=0}^\infty$ is bounded.

The following two lemmas establish the existence of an acceptable step size as well as a lower bound on accepted step sizes.

\begin{lemma} \label{lem:exact-alpha-upper}
  If $s_k$ and $\rho_k$ are chosen as in Algorithm~\ref{alg:exact} and $s_k^T H_k s_k \ge 0$,
  then there exists $\bar\alpha$ such that the sufficient decrease condition \eqref{eq:exact-merit-sd} is satisfied for $\alpha\in(0,\bar\alpha)$.
\end{lemma}
\begin{proof}
  From \eqref{eq:exact-dd-ineq} and \eqref{eq:rho}, $s_k$ is a descent direction.
  Then apply standard arguments, see, e.g., \cite[Section 3.1]{JNocedal_SJWright_2006a}.
\end{proof}

\begin{lemma} \label{lem:exact-alpha-lower}
 Let  Assumptions~\ref{ass:exact} (iii) and (v).
 If $\alpha_k$ is selected using backtracking line-search with $0<c_3\leq\alpha_k^{(0)}\leq1$,
  then there exists $\underline{\alpha}(\rho_k) > 0$ such that $\alpha_k\geq\underline{\alpha}(\rho_k)$ for all $k$.
  If $\rho_k \le \overline{\rho}$, then $\underline{\alpha}(\rho_k) \geq\underline{\alpha}(\overline{\rho}) > 0$.
\end{lemma}
\begin{proof}
      See  \cite[Lemma~5.2.3]{DSGrundvig_2025a}.
\end{proof}

It can also be shown that the penalty parameters $\rho_k$ will eventually be fixed.
\begin{lemma} \label{lem:exact-rho-bound}
  If Assumptions~\ref{ass:exact} (ii)-(iv) hold, then the steps $s_k$ and 
  the Lagrange multiplier estimates $\lambda_{k+1}$ are uniformly bounded, and
  there exists $k_0$ such that $\rho_{k+1}=\rho_k=:\bar\rho$ for all $k\geq k_0$.
\end{lemma}
\begin{proof}
    See   \cite[Lemma~5.2.4]{DSGrundvig_2025a}.
\end{proof}

The previous lemmas can now be used to prove a convergence result for the SQP Algorithm~\ref{alg:exact}.
\begin{theorem} \label{thm:exact-convergence}
  If Assumptions~\ref{ass:exact} hold, then the iterates of the line-search SQP Algorithm~\ref{alg:exact}
  satisfy 
  \[
    \lim_{k\rightarrow\infty} \|\nabla f(x_k) - c'(x_k)^T\lambda_{k+1} \| =0
    \quad \mbox{ and } \quad 
    \lim_{k\rightarrow\infty} \|c(x_k)\|_1 =0,
  \]
  which correspond to the measures of first-order criticality as defined by \eqref{eq:exact-fonc}.
\end{theorem}
\begin{proof}
  The proof is standard, but included to provide context for the following algorithmic developments.
  
  By Lemma~\ref{lem:exact-rho-bound}, for any $K> k_0$,
  \[
    \phi(x_K;\bar\rho) - \phi(x_{k_0};\bar\rho) = \sum_{k=k_0}^{K-1} \phi(x_{k+1};\bar\rho) - \phi(x_k;\bar\rho) 
    \leq \sum_{k=k_0}^{K-1} c_1\alpha_k D(\phi(x_k;\bar\rho);s_k).
  \]
  Assumption~\ref{ass:exact} (i) implies the boundedness of the left hand side so
  $ \infty > \sum_{k=k_0}^{K-1} -c_1\alpha_k D(\phi(x_k;\bar\rho);s_k)$.
  By Lemma~\ref{lem:exact-alpha-lower}, $\alpha_k$ is bounded from below which implies that $-D(\phi(x_k;\bar\rho);s_k) \rightarrow0$ as $k\rightarrow \infty$.
  This limit, along with \eqref{eq:exact-dd-ineq} and \eqref{eq:rho} implies that as $k\rightarrow\infty$
  \begin{equation} \label{eq:limit}
    -D(\phi(x_k;\bar\rho);s_k) \geq s_k^TH_ks_k + (\bar\rho -\|\lambda_{k+1}\|_\infty )\|c(x_k)\|_1 \geq s_k^TH_ks_k \rightarrow 0.
  \end{equation}
  Since $-s_k^TH_ks_k\rightarrow0$, \eqref{eq:limit} implies that $(\bar\rho -\|\lambda_{k+1}\|_\infty)\|c(x_k)\|_1\rightarrow0$.
  Noting that $\bar\rho>\|\lambda_{k+1}\|_\infty+\sigma/2$ means that $\|c(x_k)\|_1\rightarrow0$ as in \eqref{eq:exact-fonc}.
  Applying Assumption~\ref{ass:exact}(iii) to $-s_k^TH_ks_k\rightarrow0$ implies that $s_k\rightarrow0$.
  Finally $s_k\rightarrow0$, along with the first line of \eqref{eq:exact-kkt} implies that 
  $ \nabla f(x_k) - c'(x_k)^T\lambda_{k+1} \rightarrow0$ as desired.
\end{proof}

\section{$\ell_1$-Merit Function SQP with Inexact Function Information} \label{sec:inexact-merit}
Now, assume that around an iterate $x_k$, we can compute smooth models $m_k: \real^n\rightarrow\real$ of the objective and 
$h_k:\real^n\rightarrow\real^m$ of the constraints.
Furthermore, we assume that these models and their derivatives are computationally inexpensive to evaluate,
and we will use these models instead of the objective and constraint functions to compute a new iterate.
Finally, we assume that objective function and constraint function models have  error bounds
\begin{equation} \label{eq:obj-error}
  |f(x) -m_k(x)| \leq M_f e_k^f(x)  \quad \mbox{ and } \quad 
  \|c(x) -h_k(x)\|_1 \leq M_c e_k^c(x),
\end{equation}
where $M_f, M_c$ are constants, and $e_k^f$ and $e_k^c$ are computable error functions.
The reason to separate the constants  $M_f$, $ M_c$ from the functions $e_k^f$, $e_k^c$
is that often one can establish error bounds of the type \eqref{eq:obj-error} with constants 
$M_f$, $ M_c$ that dependent on bounds of inverse matrices and Lipschitz constants that can 
be proven to exist, but are difficult to compute. Our algorithm requires the existence of 
constants $M_f$, $ M_c$, but does not require their values. 
The error functions $e_k^f$ and $e_k^c$ are assumed to be computable, but they they may be large
for $x$ away from the current iterate $x_k$ at which the current models $m_k$ are $h_k$ are built,
and could even take the value $\infty$ away from the current iterate $x_k$.
However, we assume that given tolerances $\tau_k^f$ and $\tau_k^c$, that will be determined
by our algorithm, approximations $m_k$ and $h_k$ can be computed such that the error bounds at $x_k$ satisfy
\begin{equation} \label{eq:ebound-tau}
	e_k^f(x_k) \le \tau^f_k  \quad \mbox{ and } \quad 
	e_k^c(x_k) \le \tau^c_k.
\end{equation}
In addition, relative errors $\tau_k^{f,g}$ and $\tau_k^{c,g}$  on the derivatives will be described below.
The tolerances $\tau^f_k$, $\tau^c_k$, $\tau_k^{f,g}$ and $\tau_k^{c,g}$ will ultimately be specified by our optimization algorithm.

\subsection{Basic Algorithm} \label{sec:inexact-merit-basic}
For now, assume that at step $k$ of an algorithm to solve \eqref{eq:exact-nlp} models $m_k$ and $h_k$ of the objective and constraints 
around the current iterate $x_k$ have been given.
Later, specific conditions for constructing such models will be provided.
If the models $m_k$ and $h_k$ are sufficiently accurate in a region around the current iterate $x_k$, 
it is reasonable to believe that an approximate solution of 
\begin{subequations} \label{eq:inexact-nlp}
\begin{align}
  \min \quad & m_k(x), \\
  \text{s.t.}\quad & h_k(x)=0,
\end{align}
\end{subequations}
provides a better approximation of the solution of the exact problem \eqref{eq:exact-nlp} than the current iterate $x_k$.

The new iterate $x_{k+1}$ is computed as an approximate solution of \eqref{eq:inexact-nlp} in two steps.
Each step is related to the SQP method with $\ell_1$-merit function applied to \eqref{eq:inexact-nlp}. 
Each subproblem \eqref{eq:inexact-nlp} has its own merit function defined as
\begin{equation} \label{eq:nexact-merit}
  \psi_k(x;\rho)= m_k(x) + \rho \| h_k(x)\|_1.
\end{equation}
The objective and constraint error bounds \eqref{eq:obj-error} imply the following  error bound for the merit function
\begin{align}     \label{eq:merit-error}
  |\phi(x;\rho) - \psi_k(x;\rho)| &=\Big|f(x)-m_k(x) + \rho(\|c(x)\|_1 - \|h_k(x)\|_1)\Big| \nonumber \\
  &\leq   \max\{ M_f, M_c \} \big( e_k^f(x) + \rho e_k^c(x)  \big)   =: \max\{ M_f, M_c \} e_k(x;\rho). 
  \end{align}
Note that the merit function error bound $e_k(x;\rho)$ is parameterized by the penalty parameter $\rho$.

Throughout the rest of the section, the subscript $k$ will denote the outer iterations which track the progress of the algorithm to solve \eqref{eq:exact-nlp}.
The double subscript $k,j$ will denote the $j$th iteration of the iterative algorithm used to compute an approximate solution of the $k$th 
subproblem  \eqref{eq:inexact-nlp}.

In the first step towards computing the new outer iterate $x_{k+1}$, one step of the SQP method with initial iterate $x_{k,0} = x_k$ 
is applied to \eqref{eq:inexact-nlp}. If $H_{k,0}$ is a symmetric positive definite matrix that replaces the Hessian of the model Lagrangian 
\begin{equation}
  \mathcal {L}_k(x,\lambda) = m_k(x) - \lambda^Th_k(x)
\end{equation}
and the Jacobian $h_k'(x_k) \in \real^{m \times n}$ has rank $m$, then the optimality conditions corresponding to
the SQP subproblem for \eqref{eq:inexact-nlp} are given by
\begin{equation} \label{eq:nexact-kkt}
  \begin{bmatrix} H_{k,0} & -h'_k(x_k)^T\\ h'_k(x_k) & 0\end{bmatrix}
  \begin{bmatrix} s_{k,0} \\ \lambda_{k+1,0}\end{bmatrix} 
 =  -\begin{bmatrix} \nabla m_k(x_k) \\ h_k(x_k)\end{bmatrix}.
\end{equation}
As in the exact case, \eqref{eq:nexact-kkt} can be solved using an unconstrained problem, see \eqref{eq:s-prob-red} and the discussion in Section~\ref{sec:exact-merit-step}.

Continuing to follow the SQP method described in Section~\ref{sec:exact-merit}, the penalty parameter 
\begin{equation} \label{eq:nexact-merit-penalty}
    \rho_k=\rho_{k,0}
\end{equation}
is set using \eqref{eq:rho}.
The backtracking line-search is applied to the merit function \eqref{eq:nexact-merit} to compute 
a step-size $\alpha_{k,0}$ is chosen such that the sufficient decrease condition
\begin{equation} \label{eq:nexact-merit-sd}
  \psi_k(x_k +\alpha_{k,0}s_{k,0};\rho_{k,0})\leq \psi_k(x_k ;\rho_{k,0}) + c_1\alpha_{k,0} D(\psi_k(x_k ;\rho_{k,0}); s_{k,0})
\end{equation}
is satisfied.
The subiterate
\begin{equation} \label{eq:nexact-merit-xkC}
        x_k^C=x_k+\alpha_{k,0}s_{k,0}
\end{equation}
is referred to as a generalized Cauchy point.  
With some assumptions on the models $m_k$ and $h_k$ which mirror Assumptions~\ref{ass:exact}, 
the existence of such a $x_k^C$ follows easily from the results in Section~\ref{sec:exact-merit-convergence}.
We will make this precise at the end of this subsection.
This subiterate, together with the sufficient decrease condition
\eqref{eq:nexact-merit-sd}, will ensure the convergence of our algorithm.

In theory, it is possible to use $x_{k+1} = x_k^C$. However, in our applications, the generation of models $m_k$ and $h_k$
is expensive and we want to utilize these models more. Therefore, we continue to use these models and additional SQP iterations
applied to \eqref{eq:inexact-nlp} to compute $x_{k+1}$. The computation of $x_{k+1}$ given  $x_k^C$ is referred to as the second step.
We require that given $a_1\in(0,1]$, the new outer iterate $x_{k+1}$ satisfies
\begin{equation} \label{eq:xkC-cond}
  \psi_k(x_k;\rho_k) - \psi_{k+1}(x_{k+1};\rho_k) \ge a_1\left( \psi_k(x_k;\rho_k) -\psi_k(x^C_k;\rho_k) \right).
\end{equation}
Computing $x_{k+1}$ that satisfies the ``ideal'' decrease condition \eqref{eq:xkC-cond} is non-trivial, because it involves the merit function $\psi_{k+1}$
computed with the new models $m_{k+1}$  and  $h_{k+1}$, and the generation of these models has a computational cost.
Therefore, even checking whether a given $x_{k+1}$ satisfies \eqref{eq:xkC-cond} is computationally costly.
We will discuss a practical way to compute $x_{k+1}$ that satisfies \eqref{eq:xkC-cond} in Section~\ref{sec:inexact-algo-practical}.
The basic line-search SQP Algorithm \ref{alg:inexact-basic} with tunable models $m_k$ and $h_k$ summarizes our approach.

\begin{algorithm}[htp]
	\caption{Basic Line-Search SQP Algorithm} \label{alg:inexact-basic}
	\begin{algorithmic}[1]    
    \REQUIRE $c_1\in(0,1)$, $x_0$, $m_k$, $\tau_g,\tau_J,\tau_c \in (0,1)$, $\rho_0,\sigma>0$, tolerances $\mbox{tol}_f,\mbox{tol}_c>0$
    \ENSURE Point $x_K$ where $\| \nabla m_K(x_K) - h_K'(x_K)^T\lambda_K\| < \mbox{tol}_f$, $\| h_K(x_K) \|_1 < \mbox{tol}_c$.

    \FOR{$k=0,1,2,\dots$}
      \STATE Compute models $m_k$ and $h_k$.
      \STATE {\bf if} $\| \nabla m_k(x_k) - h_k'(x_k)^T\lambda_k\| < \mbox{tol}_f$ and $\| h_k(x_k) \|_1 < \mbox{tol}_c$, {\bf then} stop.
      \STATE Find $s_k = s_{k,0}$, $\lambda_{k+1,0}$ by solving \eqref{eq:nexact-kkt}.
      \STATE Compute $\rho_{k}$ using \eqref{eq:rho} with $\lambda_{k+1}$ replaced by $\lambda_{k+1,0}$.
      \STATE Use a backtracking line-search algorithm to find $\alpha_k = \alpha_{k,0}$ that satisfies sufficient decrease \eqref{eq:nexact-merit-sd}.
      \STATE Set generalized Cauchy point $x^C_k= x_k+\alpha_k s_k$.
      \STATE Compute $x_{k+1}$ such that \eqref{eq:xkC-cond} is satisfied.
	\ENDFOR
	\end{algorithmic}
\end{algorithm}

To guarantee the existence of  $x_k^C$ that satisfies the sufficient decrease condition \eqref{eq:nexact-merit-sd} 
and to ensure convergence of Algorithm~\ref{alg:inexact-basic} we need the following assumptions, which mirror Assumptions~\ref{ass:exact}.

\begin{assumptions} \label{ass:inexact}
	\begin{enumerate}[i)]
    \item The sequence $\{m_k(x_k)\}_{k=0}^\infty$ is bounded from below.
    \item The sequences $\{\|h_k(x_k)\|_1\}_{k=0}^\infty$, $\{\|h'_k(x_k)\|_1\}_{k=0}^\infty$, $\{\|\nabla m_k(x_k)\|_1\}_{k=0}^\infty$ are bounded.
    \item There exists $0 < h_l <h_u$ such that
             $h_l \| v \|_2 \le v^T H_{k,0} v \leq h_u\|v\|_2$  for all $v \in \real^n$ and all $k$.
    \item The singular values of $h'_k(x_k)$ are uniformaly bounded: $0 < \sigma_{\min}  \leq \sigma_{\min}(h'_k(x_k))\leq \sigma_{\max}(h'_k(x_k)) \leq \sigma_{\max}$.
    \item For every $k$,  the gradients $\nabla m_k$ and $\nabla (h_k)_i$, $i = 1, \ldots, m$,  are Lipschitz continuous with constants $L_k^f$ and
            $L_k^{h_i}$, $i = 1, \ldots, m$, and  Lipschitz constants $L_k^f$ and $L_k^h=\begin{bmatrix} L_k^{h_1} & \cdots & L_k^{h_m} \end{bmatrix}^T$
            are uniformly bounded from above,  $L_k^f\leq C_f$ and $\|L_k^h\|_1\leq C_h$ for all $k$.
	\end{enumerate}
\end{assumptions}

The Cauchy point and penalty parameter are selected using one step of Algorithm~\ref{alg:exact} applied with
$f$ and $c$ replaced  by $m_k$ and $h_k$.
Thus,  Lemma~\ref{lem:exact-alpha-upper} with $f$ and $c$ replaced by $m_k$ and $h_k$ implies that the sufficient decrease condition \eqref{eq:nexact-merit-sd} is satisfiable.
Since a backtracking line-search is used to select $\alpha_{k,0}$, then if Assumptions~\ref{ass:inexact} (iii) and (v) hold,
Lemma~\ref{lem:exact-alpha-lower} with $f$ and $c$ replaced  by $m_k$ and $h_k$ giving a uniform lower bound on $\alpha_{k,0}$.
Additionally, if Assumptions~\ref{ass:inexact} (ii)-(iv) hold then Lemma~\ref{lem:exact-rho-bound} with $f$ and $c$ replaced by $m_k$ and $h_k$
implies that there exists $k_0$ such that $\rho_k=\bar\rho$ for $k\geq k_0$.

The following convergence result requires some relative error conditions.
In particular, given positive constants $\tau^{f,g},\tau^{c,g},\tau^c \in (0,1)$, we assume
relative error bounds on the derivatives of the objective and constraints,
\begin{align} \label{eq:rel-error1}
  \frac{\|\nabla m_k(x_k) - \nabla f(x_k)\|}{\|\nabla m_k(x_k)-h_k'(x_k)^T\lambda_{k+1}\|} &\le \tau^{f,g}, &
  \frac{\| h_k'(x_k) - c'(x_k)\|}{\|\nabla m_k(x_k)-h_k'(x_k)^T\lambda_{k+1}\|} &\le \tau^{c,g},
\end{align}
and on the constraint function,
\begin{equation} \label{eq:rel-error2}
  \frac{\|h_k(x_k) - c(x_k)\|}{\|h_k(x_k)\|} \le \tau^c.
\end{equation}

\begin{theorem} \label{thm:inexact-fun-conv}
  Let Assumptions~\ref{ass:inexact} hold.
  If Algorithm~\ref{alg:inexact-basic} generates a sequence $\{x_k\}$ of iterates, the iterates satisfy
  \[
    \lim_{k\rightarrow\infty} \|\nabla m_k(x_k) - h_k'(x_k)^T\lambda_{k+1,0} \| =0
     \quad \mbox{ and } \quad 
    \lim_{k\rightarrow\infty} \|h_k(x_k)\|_1 =0.
  \]
  If, in addition, the relative error conditions \eqref{eq:rel-error1} and \eqref{eq:rel-error2} hold then 
  \[
    \lim_{k\rightarrow\infty} \|\nabla f(x_k) - c'(x_k)^T\lambda_{k+1,0} \| =0
     \quad \mbox{ and } \quad 
    \lim_{k\rightarrow\infty} \|c(x_k)\|_1 =0.
  \]
\end{theorem}
\begin{proof}
  The initial steps of this proof apply the proof of Theorem~\ref{thm:exact-convergence} to the sequence of models $m_k$, $h_k$.
  Lemma~\ref{lem:exact-rho-bound} with $f$ and $c$ replaced  by $m_k$ and $h_k$ yields $\rho_k = \bar\rho$ for all $k \ge k_0$.
  Hence, for any $K> k_0$,  \eqref{eq:nexact-merit-sd} implies
  \begin{align*}
    \psi_K(x_K;\bar\rho) - \psi_{k_0}(x_{k_0};\bar\rho) &= \sum_{k=k_0}^{K-1} \psi_{k+1}(x_{k+1};\bar\rho) - \psi_k(x_k;\bar\rho) \\
    &\leq \sum_{k=k_0}^{K-1} a_1\left( \psi_{k}(x_{k}^C;\bar\rho) - \psi_k(x_k;\bar\rho) \right) 
    \leq \sum_{k=k_0}^{K-1} c_1\alpha_k D(\psi_k(x_k;\bar\rho);s_k).
  \end{align*}
  Assumption~\ref{ass:inexact} (i) implies the boundedness of the left hand side so
  $ \infty > \sum_{k=k_0}^{K-1} -c_1\alpha_k D(\psi_k(x_k;\bar\rho);s_k)$.
  Lemma~\ref{lem:exact-alpha-lower}  with $f$ and $c$ replaced  by $m_k$ and $h_k$ yields $\alpha_k \ge \underline{\alpha}(\overline{\rho}) > 0$,
   which implies that $-D(\psi_k(x_k;\bar\rho);s_k) \rightarrow0$ as $k\rightarrow \infty$.
  This limit, along with inequalities on the merit function $\psi_k$ corresponding to  \eqref{eq:exact-dd-ineq} and \eqref{eq:rho} implies that as $k\rightarrow\infty$
  \begin{equation} \label{eq:limit-inexact}
    -D(\psi_k(x_k;\bar\rho);s_k) \geq s_k^TH_{k,0}s_k + (\bar\rho -\|\lambda_{k+1,0}\|_\infty )\|h_k(x_k)\|_1 \geq s_k^TH_{k,0}s_k \rightarrow 0.
  \end{equation}
  Since $-s_k^TH_{k,0}s_k\rightarrow0$, \eqref{eq:limit-inexact} implies that $(\bar\rho -\|\lambda_{k+1,0}\|_\infty)\|h_k(x_k)\|_1\rightarrow0$.
  Noting that $\bar\rho>\|\lambda_{k+1,0}\|_\infty+\sigma/2$ means that $\|h_k(x_k)\|_1\rightarrow0$ as in \eqref{eq:exact-fonc}.
  Applying Assumption~\ref{ass:inexact} (iv) to $-s_k^TH_{k,0}s_k\rightarrow0$ implies that $s_k\rightarrow0$.
  Finally $s_k\rightarrow0$, along with the first line of \eqref{eq:exact-kkt} implies that 
  $ \nabla m_k(x_k) - h_k'(x_k)^T\lambda_{k+1,0} \rightarrow0$ as desired.

  If the relative bounds \eqref{eq:rel-error1}-\eqref{eq:rel-error2} hold then 
  \begin{align*}
    & \lim_{k\rightarrow\infty} \|\nabla f(x_k) - c'(x_k)^T\lambda_{k+1,0} \| \\
    &\le  \lim_{k \rightarrow \infty}  \|\nabla f(x_k) - \nabla m_k(x_k)\| + \|\nabla m_k(x_k) - h_k'(x_k)^T\lambda_{k+1,0} \| + \| h_k'(x_k) - c'(x_k) \| \|\lambda_{k+1,0} \| \\
    &\le  \lim_{k \rightarrow \infty}  ( \tau^{f,g}  + 1 + \tau^{c,g}\|\lambda_{k+1,0}\|) \|\nabla m_k(x_k) - h_k'(x_k)^T\lambda_{k+1,0} \|
    = 0.
  \end{align*}
  and
 $\lim_{k \rightarrow \infty}   \|c(x_k)\|_1 
    \le  \lim_{k \rightarrow \infty}  \| c(x_k) - h_k(x_k)\|_1 + \| h_k(x_k)\|_1  
    \le  \lim_{k \rightarrow \infty}  ( \tau^c  + 1 ) \| h_k(x_k)\|_1 
    = 0$,
  as desired.
\end{proof}

So far, no assumptions other than relative error conditions \eqref{eq:rel-error1} and \eqref{eq:rel-error2} are made for the 
models $m_k$ and $h_k$. However, the convergence Theorem~\ref{thm:inexact-fun-conv} assumes that there
is a sequence  $\{x_k\}$ of iterates. In particular, these iterates have to satisfy  \eqref{eq:xkC-cond}.
As we have explained earlier, satisfying   \eqref{eq:xkC-cond} is difficult and will be addressed next.

\subsection{A Practical Approach} \label{sec:inexact-algo-practical}

In order to prove Theorem~\ref{thm:inexact-fun-conv}, the existence of a sequence  $\{x_k\}$ of iterates that satisfy the
``ideal" decrease condition \eqref{eq:xkC-cond} was required.
The condition \eqref{eq:xkC-cond} is difficult to use in practice because it requires the evaluation of the models  $m_{k+1}$ and $h_{k+1}$
during iteration $k$.
This challenge can be overcome by extending the ideas in 
\cite{DSGrundvig_MHeinkenschloss_2025a}, \cite{YYue_KMeerbergen_2013a} to the constrained case.
This is where the error bound \eqref{eq:merit-error}, i.e., the bounds  \eqref{eq:obj-error}, are used.
Moreover, we adopt a technique from  \cite{DPKouri_MHeinkenschloss_DRidzal_BGvanBloemenWaanders_2014a}
to handle the presence of uncomputable constants $M_f, M_c$ in the error bounds \eqref{eq:obj-error} and  \eqref{eq:merit-error}.

For now, assume that the error bound functions $e_k$  \eqref{eq:merit-error} for the merit function error satisfy
\begin{align}    \label{eq:inexact-lim-epsilon}
    \lim_{k\rightarrow\infty} e_k(x_{k+1};\rho_k) =0   \quad \mbox{ and } \quad 
    \lim_{k\rightarrow\infty} e_{k+1}(x_{k+1};\rho_k) =0.
\end{align}
Eventually, the conditions \eqref{eq:inexact-lim-epsilon} will be enforced algorithmically, but for clarity of analysis
they will be taken as assumptions for the time being.
Given  $\omega\in(0,1)$, the bound \eqref{eq:merit-error}  implies
\[
     |\phi(x_{k+1};\rho_k) - \psi_k(x_{k+1};\rho_k)|  
     \le  \max\{ M_f, M_c \} \, e_k(x_{k+1};\rho_k)^{(1-\omega)}   e_k(x_{k+1};\rho_k)^\omega.
\]
If \eqref{eq:inexact-lim-epsilon} holds, for large enough $k$ it follows that $ \max\{ M_f, M_c \} e_k(x;\rho)^{(1-\omega)}  \le 1$, which in turn implies
\begin{equation} \label{eq:merit-error-eps-k-omega-bnd}
     |\phi(x_{k+1};\rho_k) - \psi_k(x_{k+1};\rho_k)|    \le  e_k(x_{k+1};\rho_k)^\omega.
\end{equation}
A similar argument implies that for $k$ large enough
\begin{equation} \label{eq:merit-error-eps-k1-omega-bnd}
  |\phi(x_{k+1};\rho_k) - \psi_{k+1}(x_{k+1};\rho_k)|    \le  e_{k+1}(x_{k+1};\rho_k)^\omega.
\end{equation}

The merit function error bounds \eqref{eq:merit-error-eps-k-omega-bnd},  \eqref{eq:merit-error-eps-k1-omega-bnd}  imply that for $k$ large enough
\begin{align*}
  & \psi_k(x_k;\rho_k) - \psi_{k+1}(x_{k+1};\rho_k) \\
  &= \psi_k(x_k;\rho_k) -\psi_k(x_{k+1};\rho_k) +\psi_k(x_{k+1};\rho_k) - \phi(x_{k+1};\rho_k) +\phi(x_{k+1};\rho_k)- \psi_{k+1}(x_{k+1};\rho_k)\\
  &\ge \psi_k(x_k;\rho_k) -\psi_k(x_{k+1};\rho_k) -e_{k}(x_{k+1};\rho_k)^\omega -e_{k+1}(x_{k+1};\rho_k)^\omega.
\end{align*}
Thus, \eqref{eq:xkC-cond} is implied by 
\begin{equation} \label{eq:nexact-sc}
  -e_{k+1}(x_{k+1};\rho_k)^\omega -e_{k}(x_{k+1};\rho_k)^\omega +\psi_k(x_k;\rho_k) -\psi_k(x_{k+1};\rho_k)
  \geq a_1 \left( \psi_{k}(x_{k};\rho_k) - \psi_k(x^C_k;\rho_k) \right).
\end{equation}
A rearrangement of \eqref{eq:nexact-sc} leads to 
\begin{align} \label{eq:new-model}
	e_{k+1}(x_{k+1};\rho_k)^\omega
	  \leq & -e_k(x_{k+1};\rho_k)^\omega -\psi_k(x_{k+1}; \rho_k) +\psi_k(x_k^C; \rho_k) \nonumber \\
	         & + (1-a_1)\left(\psi_k(x_k; \rho_k) - \psi_k(x_k^C; \rho_k)\right).
\end{align}
Condition \eqref{eq:new-model} implies \eqref{eq:xkC-cond}.
As long as the right-hand side of \eqref{eq:new-model} is positive, the condition is implementable since it sets an error tolerance for the model 
constructed at the $(k+1)$st iteration.

The positivity of the right-hand side of \eqref{eq:new-model} can be ensured by requiring the following two conditions:
\begin{equation} \label{eq:simple-cond}
  \psi_k(x_{k+1}; \rho_k)\leq \psi_k(x_k^C; \rho_k)
\end{equation}
and
\begin{equation} \label{eq:mod-constraint}
  e_k(x_{k+1};\rho_k)^\omega \leq a_2(1-a_1)\left(\psi_k(x_k; \rho_k) - \psi_k(x_k^C; \rho_k)\right),
\end{equation}
for $a_2\in (0,1]$.
Since both \eqref{eq:simple-cond} and \eqref{eq:mod-constraint} are independent of the new models $m_{k+1}$ and $h_{k+1}$, 
they can be enforced and verified during the $k$th iteration. We will describe this next.

First, recall that throughout this discussion, it has been assumed that \eqref{eq:inexact-lim-epsilon} holds.
The left limit in \eqref{eq:inexact-lim-epsilon} is enforced by requiring that, in addition to  \eqref{eq:mod-constraint},
\begin{equation} \label{eq:inexact-rk-bnd}
    e_k(x_{k+1};\rho_k)^\omega  \leq r_k
\end{equation}
is satisfied
where $\{r_k\}_{k=0}^\infty$ is a sequence such that
\begin{equation} \label{eq:inexact-rk}
  \lim_{k\rightarrow\infty} r_k=0.
\end{equation}
Similarly, the right limit in \eqref{eq:inexact-lim-epsilon} is enforced by requiring that,  
\begin{equation} \label{eq:ls:ek1-rk-bnd}
  e_{k+1}(x_{k+1};\rho_k)^\omega  \leq r_k.
\end{equation}
Note that a similar forcing sequence $\{r_k\}$ was introduced in the inexact trust-region method found in \cite{DPKouri_MHeinkenschloss_DRidzal_BGvanBloemenWaanders_2014a}.

The conditions \eqref{eq:mod-constraint}, \eqref{eq:inexact-rk-bnd} will be added as an additional constraint to the inner optimization subproblem \eqref{eq:inexact-nlp}
used to compute $x_{k+1}$.  Specifically,  $x_{k+1}$ will be chosen as an approximate solution to
\begin{subequations} \label{eq:nexact-subprob}
\begin{align}
  \min \quad &m_k(x),   \label{eq:nexact-subprob-a} \\
  \text{s.t.}\quad &h_k(x)=0,   \label{eq:nexact-subprob-b} \\
  &e_k(x;\rho_k)^\omega \leq \min\{ r_k, \, a_2(1-a_1)\left(\psi_k(x_k; \rho_k) - \psi_k(x_k^C; \rho_k)\right) \}. \label{eq:nexact-subprob-extra}
\end{align}
\end{subequations}
This subproblem does not need to be solved exactly, and we will later discuss how we compute an approximate solution $x_{k+1}$
of \eqref{eq:nexact-subprob}.

The final condition left to be enforced is \eqref{eq:simple-cond}.
If the generalized Cauchy point $x_k^C$ satisfies \eqref{eq:nexact-subprob-extra}, then \eqref{eq:simple-cond} is satisfied
for $x_{k+1} = x_k^C$ and any approximate solution that gives a lower merit function value than $x_k^C$ and satisfies \eqref{eq:nexact-subprob-extra}.
Thus, we require that the generalized Cauchy point $x_k^C$ satisfies \eqref{eq:nexact-subprob-extra}, i.e., that
\begin{equation} \label{eq:cauchy-cond}
  e_k(x_k^C;\rho_k)^\omega \leq \min\{ r_k, \, a_2(1-a_1)\left(\psi_k(x_k; \rho_k) - \psi_k(x_k^C; \rho_k)\right) \}.
\end{equation}
Once a candidate generalized Cauchy point $x_k^C$ has been computed as before, we check whether it
satisfies \eqref{eq:cauchy-cond}. If not, the models $m_k$ and $h_k$ have to be refined, the
generalized Cauchy point $x_k^C$ has to be re-computed, and the process is repeated.
Since the algorithm described in this section is model agnostic, additional details on refinement will not be discussed here.
For some conditions under which it is possible to prove the existence of an acceptable Cauchy point in the unconstrained case, see \cite[Sec.~3.2]{DSGrundvig_MHeinkenschloss_2025a}.

If a generalized Cauchy point $x_k^C$ that
satisfies \eqref{eq:cauchy-cond} is computed, then $x_{k+1}$ is computed using an SQP method applied to \eqref{eq:nexact-subprob-a},  \eqref{eq:nexact-subprob-b},
started at  $x_k^C$. The constraint  \eqref{eq:nexact-subprob-extra} is not used in the SQP step computation, but simply enforced using a simple backtracking.
Algorithm~\ref{alg:subprob} on the next page describes the SQP approach used to compute $x_{k+1} = x_{k,j}$. 
Specifically, at step $j$ of solving the subminimization problem, a descent direction $s_{k,j}$ as well as a penalty parameter $\rho_{k,j}$
are generated by the SQP method applied to \eqref{eq:nexact-subprob-a},  \eqref{eq:nexact-subprob-b}.
In addition to the backtracking done to satisfy sufficient decrease \eqref{eq:exact-merit-sd}, the step-size $\alpha_{k,j}$ will be reduced until \eqref{eq:mod-constraint} is satisfied
(line \ref{alg:subprob:backtrack}).
Some additional decrease verification is also used when the penalty parameter is increased (see line~\ref{alg:subprob:rho-check}).
Algorithm~\ref{alg:subprob} is guaranteed to find $x_{k+1}$ that satisfies \eqref{eq:mod-constraint} because the initial iterate $x_{k,0}=x_k^C$ 
satisfies \eqref{eq:mod-constraint} and $\psi_k(x_{k,j}; \rho_k) \le \psi_k(x_k^C; \rho_k)$ is enforced.

\begin{algorithm}[!htb]
	\caption{Subminimization SQP Algorithm} \label{alg:subprob}
	\begin{algorithmic}[1]    
    \REQUIRE $x_k^C$, $\lambda_{k,0}$, $\rho_k$, $c_1\in(0,1)$,  $\sigma>0$, tolerances $\mbox{tol}_f,\mbox{tol}_c>0$.
    \ENSURE Point $x_{k,j}$ where $\| \nabla m_k(x_{k,j}) - h_k'(x_{k,j})^T\lambda_{k+1,j}\| < \mbox{tol}_f$, $\| h_k(x_{k,j}) \|_1 < \mbox{tol}_c$.
    \STATE Set $x_{k,0}= x_k^C$ and $\rho_{k,0}=\rho_k$.
    \FOR{$j=0, 1,2,3,\dots$}
      \STATE {\bf if }$\| \nabla m_k(x_{k,j}) - h_k'(x_{k,j})^T\lambda_k\| < \mbox{tol}_f$ and $\| c(x_k) \|_1 < \mbox{tol}_c$, {\bf then} stop.
      \STATE Find $s_{k,j}$, $\lambda_{k+1,j}$ by solving \eqref{eq:exact-kkt} with $H_k, f(x_k), c(x_k), c'(x_k)$ replaced by $H_{k,j}, m_k(x_{k,j}), h_k(x_{k,j}), h_k'(x_{k,j})$.
      \STATE Compute $\rho_{k,j}$ using \eqref{eq:rho} with $\lambda_{k+1}, \rho_{k-1}$ replaced by $\lambda_{k+1,j}, \rho_{k,j-1}$.
      \STATE Use backtracking line-search to find $\alpha_{k,j}$ that satisfies the sufficient decrease condition 
                    $ \psi_k(x_k +\alpha_{k,j}s_{k,j};\rho_{k,j})\leq \psi_k(x_k ;\rho_{k,j}) + c_1\alpha_{k,j} D(\psi_k(x_k ;\rho_{k,j}); s_{k,j})$ and 
                     \eqref{eq:nexact-subprob-extra}.      \label{alg:subprob:backtrack}
      \STATE Set $x_{k,j+1}= x_{k,j}+\alpha_{k,j} s_{k,j}$.
     \STATE {\bf if } $\rho_{k,j}>\rho_k$ and $\psi_k(x_{k,j}; \rho_k)> \psi_k(x_k^C; \rho_k)$, {\bf then} set $x_{k,j+1}=x_{k,j}$ and stop.  \label{alg:subprob:rho-check}
     \ENDFOR
	\end{algorithmic}
\end{algorithm}

Finally, once $x_{k+1}$ has been computed, new models $m_{k+1}$, $h_{k+1}$ are generated such that the associated
merit function error bound $e_{k+1}(x_{k+1};\rho_k)$ satisfies \eqref{eq:new-model} and \eqref{eq:ls:ek1-rk-bnd}, i.e., satisfies
\begin{align} \label{eq:new-model-rev}
	e_{k+1}(x_{k+1};\rho_k)^\omega
	  \leq  \min\big\{ r_k, & -e_k(x_{k+1};\rho_k)^\omega -\psi_k(x_{k+1}; \rho_k) +\psi_k(x_k^C; \rho_k) \nonumber \\
	                                 & + (1-a_1)\left(\psi_k(x_k; \rho_k) - \psi_k(x_k^C; \rho_k)\right) \big\}.
\end{align}

All of the work in this section can now be combined into a complete Line-Search SQP Algorithm~\ref{alg:inexact}.
The inner loop over $i$ enforces that the generalized Cauchy point $x_k^C$ satisfies \eqref{eq:cauchy-cond}.
Given model $m_k$ and $h_k$ we compute a generalized Cauchy point $x_k^C$. If it violates  \eqref{eq:cauchy-cond},
the models $m_k$ and $h_k$ are refined by incorporating information at the current $x_k^C$. 
Instead of refining $m_k$ and $h_k$, which has an additional cost, one can first perform an additional
backstep in Step~\ref{alg:inexact-additional-backstep} of Algorithm~\ref{alg:inexact} to compute  $x_k^C$, as done in the unconstrained case
\cite{DSGrundvig_MHeinkenschloss_2025a}.

\begin{algorithm}[htp]
	\caption{Line-Search SQP Algorithm} \label{alg:inexact}
	\begin{algorithmic}[1]    
    \REQUIRE $x_0$, $c_1\in(0,1)$,  $a_2\in(0,1]$, $a_1,\delta\in(0,1)$, $\tau_0 > 0$,
    $\tau^{f,g},\tau^{c,g},\tau^c \in (0,1)$, $\rho_0,\sigma>0$, $\omega \in (0,1)$, $\{r_k\}_{k=0}^\infty$ with $\lim_{k\rightarrow\infty} r_k=0$,
    tolerances $\mbox{tol}_f,\mbox{tol}_c>0$
    \ENSURE Point $x_K$ where $\| \nabla m_K(x_K) - h_K'(x_K)^T\lambda_{K+1}\| < \mbox{tol}_f$, $\| h_K(x_K) \|_1 < \mbox{tol}_c$.

    \FOR{$k=0,1,2,\dots$}
      \STATE {\bf if} $\| \nabla m_k(x_k) - h_k'(x_k)^T\lambda_{k+1}\| < \mbox{tol}_f$ and $\| h_k(x_k) \|_1 < \mbox{tol}_c$. {\bf then} stop.
      \FOR{$i=0,1,\dots$} \label{alg:inexact-loop}
        \IF{$i=0$}
          \STATE Construct models  $m_k$ and $h_k$ that satisfy $e_k(x_k;\rho_k) \leq \tau_k$ and \eqref{eq:rel-error1}, \eqref{eq:rel-error2}.
                   \label{alg:inexact-model-construct}
        \ELSE
          \STATE Construct refined models  $m_k$ and $h_k$ that satisfy $e_k(x_k;\rho_k) \leq \tau_k$ and  \eqref{eq:rel-error1}, \eqref{eq:rel-error2}.
                 \label{alg:inexact-model-refine}
        \ENDIF
        \STATE Find $s_{k,0}$, $\lambda_{k+1}$ by solving \eqref{eq:nexact-kkt}.
        \STATE Set $\rho_k$ using \eqref{eq:rho} with $\lambda_{k+1}$ replaced by $\lambda_{k+1,0}$.
        \STATE Find $\alpha_{k,0}$ that satisfies the sufficient decrease condition \eqref{eq:nexact-merit-sd}.
                     Set $x^C_k= x_k+\alpha_{k,0} s_{k,0}$.
        \STATE If $x_k^C= x_k+\alpha_{k,0} s_{k,0}$ satisfies \eqref{eq:cauchy-cond}, goto line~\ref{alg:inexact-cont}.     \label{alg:inexact-additional-backstep}
      \ENDFOR
      \STATE Find $x_{k+1}$ using Algorithm~\ref{alg:subprob}.     \label{alg:inexact-cont}
      \STATE Set \begin{align*}
        \tau_{k+1} = \min\big\{ r_k, &  -e_k(x_{k+1};\rho_k)^\omega -\phi_k(x_{k+1};\rho_k) +\phi_k(x_k^C;\rho_k) \\
                                              &+ (1-a_1)\left(\phi_k(x_k;\rho_k) - \phi_k(x_k^C;\rho_k)\right) \big\}.
      \end{align*} \label{alg:inexact-new-err}
		\ENDFOR
	\end{algorithmic}
\end{algorithm}

\clearpage

Corollary~\ref{cor:inexact-fun-conv} gives the final convergence result for the practical Algorithm~\ref{alg:inexact}.

\begin{corollary} \label{cor:inexact-fun-conv}
  Let Assumptions~\ref{ass:inexact} hold.
  If the number of times the model construction loop in line~\ref{alg:inexact-loop} of Algorithm \ref{alg:inexact}
  is bounded uniformly in $k$, 
  then the iterates calculated by the line-search in Algorithm~\ref{alg:inexact} satisfy
  \[
    \lim_{k\rightarrow\infty} \|\nabla f(x_k) - c'(x_k)^T\lambda_{k+1,0} \| =0
     \quad \mbox{ and } \quad 
    \lim_{k\rightarrow\infty} \|c(x_k)\|_1 =0.
  \]
 \end{corollary}
\begin{proof}
  The iterates $x_k$ generated by Algorithm~\ref{alg:inexact} satisfy conditions \eqref{eq:simple-cond} and \eqref{eq:mod-constraint},
  which together imply \eqref{eq:new-model}, which in turn implies  \eqref{eq:xkC-cond}.
  Thus, for large enough $k$, the condition \eqref{eq:xkC-cond} is satisfied.
  In addition, the iterates $x_k$ also satisfy the relative error conditions \eqref{eq:rel-error1}, \eqref{eq:rel-error2}.
  Therefore, the iterates generated by Algorithm~\ref{alg:inexact} satisfy Algorithm~\ref{alg:inexact-basic} and the conditions of
  Theorem~\ref{thm:inexact-fun-conv}.  Thus, the desired result follows from Theorem~\ref{thm:inexact-fun-conv}.
\end{proof}

Since Algorithm~\ref{alg:inexact} is agnostic of how models are computed and refined, the assumption that the
number of times the model construction loop in line~\ref{alg:inexact-loop} of Algorithm~\ref{alg:inexact} is bounded uniformly in $k$
is needed in Corollary~\ref{cor:inexact-fun-conv}.
 For the important case where models are generated from ROMs as outlined in the introduction (see \eqref{eq:exact-nlp-implicit-constraints}
and \eqref{eq:exact-nlp-implicit-constraints-models}) and used in Section~\ref{sec:numerics}, one can prove that this assumption
holds. For the unconstrained case see \cite[Sec.~3.2]{DSGrundvig_MHeinkenschloss_2025a} and the constrained case see  \cite[Sec.~5.3.3]{DSGrundvig_2025a}.


\section{Numerical Results} \label{sec:numerics}

This section considers a boundary control problem governed by the Boussinesq equations.
The unconstrained version of this boundary control problem was studied in  \cite{KIto_SSRavindran_1998b},
with emphasis on finite element discretization of the optimal control problem.
This section uses a change in notation.
The optimization variables are the control $u$ and its discretization $\bu$, whereas $x$ denotes a point in
the domain $\Omega$ or on its boundary.

\subsection{Boundary Control of the Boussinesq PDE}
Let $\Omega=(0,8)\times(0.5,1) \cup (1,8)\times(0,0.5)$ be the domain given in Figure~\ref{fig:bous-grid}.
Let $\Gamma_{\text{in}}=\{0\}\times(0.5,1)$ be the inflow boundary, $\Gamma_{\text{out}}=\{8\}\times(0,1)$ the outflow boundary,
$\Gamma_c=((0,1)\times\{0.5\})\cup(\{1\}\times(0,0.5))$ the channel boundary, and $\Gamma_d=\partial\Omega\setminus(\Gamma_{\text{in}}\cup\Gamma_{\text{out}}\cup\Gamma_c)$ the control boundary.
See Figure~\ref{fig:bous-grid} for an illustration.
\begin{figure}[!htb]
	\centering
	\includegraphics[width=0.85\textwidth]{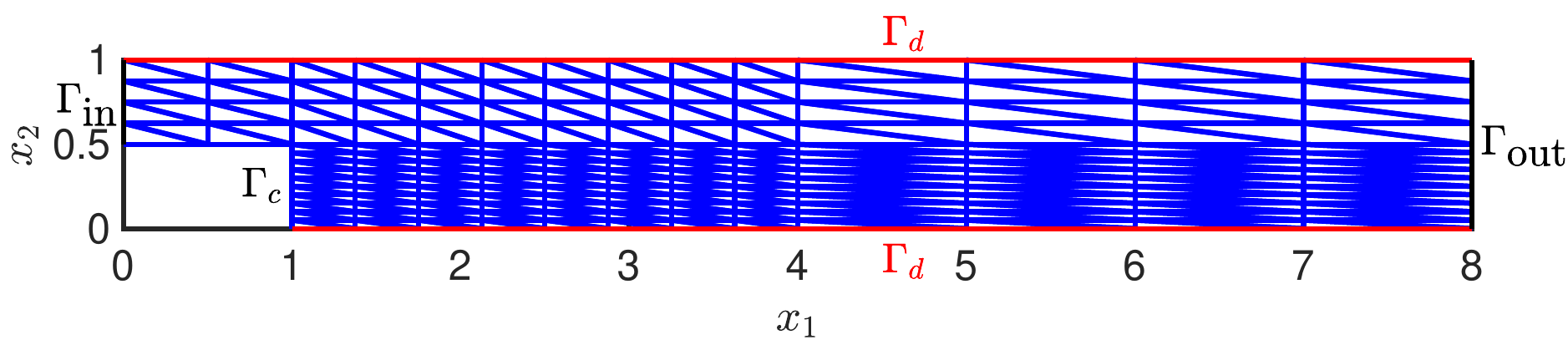}
  \caption{Geometry of the backstep region with corresponding (coarse) finite element grid.}
  \label{fig:bous-grid}
\end{figure}

Given a temperature control $u$ on $\Gamma_d$, the velocity $v$, temperature $T$, and pressure $p$ are solutions to the nondimensional Boussinesq equations
\begin{subequations} \label{eq:boussi-eq}
  \begin{align} 
    -\frac{1}{Re} \Delta v(x) +(v(x)\cdot\nabla)v(x) +\nabla p(x) +\frac{Gr}{Re^2}T(x)\bg&=0 &x\in\Omega,\\
    \nabla\cdot v(x)&=0 &x\in\Omega,\\
    -\frac{1}{RePr}\Delta T(x) +v(x)\cdot \nabla T(x) &=0 &x\in\Omega,\\
    v(x)=0, \quad \frac{\partial T}{\partial n}(x) + T(x)&=u(x) &x\in\Gamma_d, \label{eq:ecmod:BS:boussi-eq-cntrl} \\
    v(x)=0,\quad  T(x)&=1 &x\in\Gamma_c,\label{eq:ecmod:BS:ns-eqd}\\
    v(x)=d_{\text{in}}(x),\quad T(x)&=0 &x\in\Gamma_{\text{in}},\\
    v(x)=d_{\text{out}}(x),\quad \frac{\partial T}{\partial n}(x)&=0 &x\in\Gamma_{\text{out}}.
  \end{align}
\end{subequations}
In \eqref{eq:boussi-eq} $Re$ is the Reynolds number, $Gr$ the Grashof number, $Pr$ the Prandtl number, $n$ is the outward unit normal,
$d_{\text{in}}$, $d_{\text{out}}$ are given Dirichlet velocity boundary data, set to
\[
  d_{\text{in}}(x) = \big(8(x_2-0.5)(1-x_2), \; 0 \big)^T, \quad \mbox{ and } \quad
  d_{\text{out}}(x) = \big((1-x_2)x_2, \; 0 \big)^T,
\]
and $\bg=(0,-1)^T$ is the nondimensionalized gravitational acceleration. 

We want to find a control $u \in L^2(\Gamma_d)$ that minimizes recirculation in the subdomain $D=[1,3]\times[0,0.5]$ behind the step.
Recirculation is quantified using the integral of the vorticity squared. Moreover, we add a penalty term for the control.
Thus, the unconstrained optimal control problem is
\begin{equation} \label{eq:boussi:unc-opt}
    \min_{u\in L^2(\Gamma_d) }\quad \frac{1}{2} \int_D \big(\partial_{x_1}v_2(x;u) -\partial_{x_2}v_1(x;u) \big)^2 dx +\frac{\omega}{2}\int_{\Gamma_d} u^2(x) dx,
\end{equation}
where the velocity $v(\cdot; u)$, the pressure $p(\cdot; u)$, and the temperature $T(\cdot; u)$ solve \eqref{eq:boussi-eq}.
For the constrained version, we include a constraint on the temperature. Specifically, given the subdomain $\Omega_T= [1,8]\times [0,0.5]$ behind the step, 
\begin{equation} \label{eq:boussi:const-opt}
    \min_{u\in L^2(\Gamma_d) } \; \frac{1}{2} \int_D \big(\partial_{x_1}v_2(x;u) -\partial_{x_2}v_1(x;u) \big)^2 dx +\frac{\omega}{2}\int_{\Gamma_d} u^2(x) dx, \qquad
    \text{s.t.} \;\; \int_{\Omega_T} T^2(x;u) dx = T_d,
\end{equation}
where again the velocity $v(\cdot; u)$, the pressure $p(\cdot; u)$, and the temperature $T(\cdot; u)$ solve \eqref{eq:boussi-eq}.
Given Reynolds number $Re$, Grashof number $Gr$, and $Pr$ Prandtl number, the prescribed temperature $T_d$ is
set to be $\widehat T_d/2$,
where $\widehat T_d =  \int_{\Omega_T} T^2(x;u_{\text{unc}})dx$ is computed using optimal temperature for the unconstrained problem \eqref{eq:boussi:unc-opt}.
The control penalty parameter is
\[
  \omega = 10^{-2}.
\]

\subsection{Discretization}
The optimal control problem is discretized using a Taylor-Hood finite element approximation. See, e.g., \cite{HCElman_DJSilvester_AJWathen_2014a}.
Specifically, the finite element mesh is a uniform refinement of the mesh shown in Figure~\ref{fig:bous-grid}.
The mesh used has either 3,168 ($Re=$ 10) or 8,800 ($Re=$ 100 and 200) triangles.
Piecewise quadratic elements are used for the velocity and the temperature, piecewise linear elements are used for the pressure.

The boundary control is approximated by 
\begin{equation} \label{eq:ecmod:BS:cont-fe}
  u_h(x)= \sum_{j=1}^{n_u} \bu^{(j)} \varphi_j(x),
\end{equation}
where $n_u=34$ and $\varphi_j$ are the control basis functions.
The basis functions are shifted versions of the following two functions
\[
  H_1(x_1) :=\begin{cases}
    (x_1+1)^2(1-2x_1) & -1<x_1\leq 0,  \\
    (1-x_1)^2(1+2x_1) & 0<x_1\leq 1, \end{cases} \qquad
  H_2(x_1) :=\begin{cases}
    x_1(x_1+1)^2 & -1<x_1\leq 0,  \\
    x_1(1-x_1)^2 & 0<x_1\leq 1. \end{cases}
\]
The first 18 basis functions correspond to the top boundary $[0,8]\times\{1\}$, with one each of $H_1$ and $H_2$ centered at the grid points $x_1=0,\ldots,8$,
\[
  \varphi_{2j+1}(x)= H_1(x_1-j), \quad \varphi_{2j+2}(x)= H_2(x_1-j) \qquad x\in [0,8]\times\{1\},\quad j=0,\ldots,8.
\]
The remaining 16 basis functions correspond to the bottom boundary $[1,8]\times\{0\}$
\[
  \varphi_{2j+17}(x)= H_1(x_1-j), \quad \varphi_{2j+18}(x)= H_2(x_1-j) \qquad x\in [1,8]\times\{0\},\quad j=1,\ldots,8.
\]

The discretization of the Boussinesq equations \eqref{eq:boussi-eq} is given by a system
\begin{equation} \label{eq:boussi-eq-disc}
     \BR(\by;\bu) =  \BR\big((\bv, \bp, \BT)^T;\bu\big) = \bzero,
\end{equation}
where $\bu \in \real^{n_u}$, $n_u = 34$, are the discretized controls,
$\by = (\bv, \bp, \BT)^T \in \real^N$, where $N=$ 21,184 ($Re=$ 10) or $N=$ 58,184 ($Re=$ 100 and 200), are the discretized velocities, pressure, and temperature.

The discretization of the unconstrained problem  \eqref{eq:boussi:unc-opt} is given by
\begin{equation} \label{eq:boussi:unc-opt-disc}
   \min_{\bu \in \real^{n_u} } \quad     f(\bu) :=  \frac{1}{2} \Big\| \BQ  \big( \BI_\cB \bd +  \BI_\cF \bv(\bu)\big) \Big\|_2^2  +\frac{\omega}{2}\bu^T\BG\bu,
\end{equation}
where $\by(\bu) = \big( \bv(\bu),\, \bp(\bu),\, \BT(\bu) \big)^T \in \real^N$ solves \eqref{eq:boussi-eq-disc}.
The term $\BI_\cB \bd +  \BI_\cF \bv(\bu)$ is due to the fact that the velocities at the boundary are specified and $\bv(\bu)$
represents the velocities at interior nodes. For details, see \cite[Sec.~4.3]{DSGrundvig_2025a}. 
The discretization of the constrained problem  \eqref{eq:boussi:const-opt} is given by
\begin{subequations} \label{eq:boussi:const-opt-disc}
  \begin{align} 
   \min_{\bu \in \real^{n_u} } \quad   &  f(\bu) :=  \frac{1}{2} \Big\| \BQ  \big( \BI_\cB \bd +  \BI_\cF \bv(\bu)\big) \Big\|_2^2  +\frac{\omega}{2}\bu^T\BG\bu,  \\
   \text{s.t.} \quad &   c(\bu) :=  ( \BI_{\cB_T} \mathbf{1} +  \BI_{\cF_T} \BT(\bu))^T\BK( \BI_{\cB_T} \mathbf{1} +  \BI_{\cF_T} \BT(\bu)) - T_d  = 0,   \label{eq:boussi:const-opt-disc-b}
  \end{align}
\end{subequations}
where, again, $\by(\bu) = \big( \bv(\bu),\, \bp(\bu),\, \BT(\bu) \big)^T \in \real^N$ solves \eqref{eq:boussi-eq-disc}.
The term $\BI_{\cB_T} \mathbf{1} +  \BI_{\cF_T} \BT$ in \eqref{eq:boussi:const-opt-disc-b} is due to the fact that the temperature on $\Gamma_c$ is one  and $\BT(\bu)$
represents the temperature at nodes not on $\Gamma_{\text{in}} \cup \Gamma_c$. For details, see \cite[Sec.~4.3]{DSGrundvig_2025a}. 

\subsection{ROM Construction}
We use a Galerkin ROM. In the $k$-th outer iteration, the FOM \eqref{eq:boussi-eq-disc} is approximated by the ROM
\begin{equation} \label{eq:boussi-eq-disc-ROM}
    \BV_k^T  \BR( \BV_k \by;\bu) =  \BV_k^T \BR\big(\BV_k (\bv, \bp, \BT)^T;\bu\big) = \bzero.
\end{equation}
The ROM matrix $\BV_k$ has block diagonal structure with diagonal blocks
$\BV^y\in\real^{|\cF_y|\times r_y}$, $\BV^p\in\real^{n_p\times r_p}$,  $\BV^t\in\real^{|\cF_T|\times r_T}$ corresponding
to velocity, pressure, and temperature components of the state.

To construct a new model at step $k$ of the ROM based optimization algorithm (see line~\ref{alg:inexact-model-construct} in Algorithm~\ref{alg:inexact}),
information from the FOM at the current iterate $\bu_k$ is used.
The FOM state equation \eqref{eq:boussi-eq-disc}  is solved for $\by(\bu_k) , \bp(\bu_k), \BT(\bu_k)$.
In addition, the FOM adjoint equation is solved for $\blambda_\by(\bu_k) , \blambda_\bp(\bu_k) , \blambda_\BT(\bu_k)$
and the FOM sensitivities $\by_\bu(\bu_k), \bp_\bu(\bu_k), \BT_\bu(\bu_k)$ are computed.
Since different solutions have different scales, they are orthonormalized individually and then orthonormalized again to 
compute
\begin{align*} 
    \widetilde\BV^y &= \text{orth}\Big( \big[\text{orth}([\by(\bu_1) , \ldots,  \by(\bu_k)]),\, 
                                                                 \text{orth}([\blambda_\by(\bu_1) , \ldots,  \blambda_\by(\bu_k)]),\,  
                                                                  \text{orth}(\by_\bu(\bu_k)) \big]\Big), \\
  \BV^p&= \text{orth}\Big( \big[\text{orth}([\bp(\bu_1) , \ldots,  \bp(\bu_k)]),\, 
                                                \text{orth}([\blambda_\bp(\bu_1) , \ldots,  \blambda_\bp(\bu_k)]),\, 
                                                \text{orth}(\bp_\bu(\bu_k)) \big]\Big),  \\
  \BV^t&= \text{orth}\Big( \big[\text{orth}([\BT(\bu_1) , \ldots,  \BT(\bu_k)]),\, 
                                        \text{orth}([\blambda_\BT(\bu_1) , \ldots,  \blambda_\BT(\bu_k)]),\, 
                                        \text{orth}(\BT_\bu(\bu_k)) \big]\Big).
\end{align*}
To ensure stability of the ROM, the velocity ROM is enriched following \cite{GRozza_KVeroy_2007a}.
See \cite[Secs.~2.8.3, 5.5.1]{DSGrundvig_2025a} for implementation details. The velocity ROM is
$\BV^y= \text{orth}\Big(\big[ \widetilde\BV^y, \BV^y_{\text{enr}}\big]\Big)$.


\subsection{Results}
We apply the ROM model based SQP Algorithm~\ref{alg:inexact} and the  FOM only Algorithm~\ref{alg:exact} 
to solve the unconstrained problem~\eqref{eq:boussi:unc-opt-disc} and the constrained problem~\eqref{eq:boussi:const-opt-disc}.
For all constrained problems, the initial Lagrange multiplier estimate was set to $\lambda_0=\big(c'(\bu_0)^T c'(\bu_0)\big)^{-1} c'(\bu_0)\nabla f(\bu_0)$
(FOM only Algorithm~\ref{alg:exact}) or 
$\lambda_0=\big(h_0'(\bu_0)^T h_0'(\bu_0)\big)^{-1} h_0'(\bu_0)\nabla f(\bu_0)$ (ROM model based SQP Algorithm~\ref{alg:inexact}).
The control was initialized to zero, $\bu_0 = \bzero$, for all runs.
The matrix $H_k$, which replaces the Hessian $\nabla_{\bu\bu}^2\mathcal{L}$, is replaced with a modified Gauss-Newton Hessian approximation in order to ensure
positive definiteness. Specifically, for the unconstrained problem~\eqref{eq:boussi:unc-opt-disc},
$H_k =  \bv_{\bu}(\bu_k)^T \BI_\cF^T \BQ^T \BQ  \BI_\cF \bv_{\bu}(\bu_k)   + \omega \BG$, where $\bv_{\bu}(\bu_k)$ is the sensitivity
of the discretized velocity with respect to the control.
\sloppy
For the constrained problem~\eqref{eq:boussi:const-opt-disc}.
$H_k =  \bv_{\bu}(\bu_k)^T \BI_\cF^T \BQ^T \BQ  \BI_\cF \bv_{\bu}(\bu_k) - 2 \lambda_k  \BT_{\bu}(\bu_k))^T  \BI_{\cF_T}^T\BK  \BI_{\cF_T} \BT_{\bu}(\bu_k)  + \omega \BG$,
if $\lambda_k < 0$, and
$H_k =  \bv_{\bu}(\bu_k)^T \BI_\cF^T \BQ^T \BQ  \BI_\cF \bv_{\bu}(\bu_k)   + \omega \BG$, if $\lambda_k \ge 0$.
Here $\bv_{\bu}(\bu_k)$ is the sensitivity of the discretized velocity with respect to the control, and $\BT_{\bu}(\bu_k)$ is the sensitivity of the discretized 
velocity with respect to the control. For the ROM problems, the Hessian approximations are defined analogously, with 
$\bv_{\bu}(\bu_k)$ and $\BT_{\bu}(\bu_k)$ replaced by their ROM approximations.
We use matrix-free methods, and these Hessians $H_k$ are never explicitly computed.
Hessian $H_k$ times vector products can be computed efficiently using linearized state and adjoint solves. 
See \cite[Sec.~4.3]{DSGrundvig_2025a} for implementation details.
For the SQP algorithm, the KKT system \eqref{eq:exact-kkt} is solved using the projection onto the equivalent unconstrained problem 
\eqref{eq:s-prob-red}, and the latter is solved using the conjugate gradient method, which only requires matrix $P_k H_k P_k $ times vector
products; the matrices $P_k$ and $H_k$ are not formed explicitly.
See the discussion at the end of Section~\ref{sec:exact-merit-step}.

Table~\ref{tb:exact-comp} below summarizes the algorithmic performance of the various methods.
Four different parameter sets were used;
in all experiments, the Prandtl number was fixed at $Pr=0.72$,
while the Reynolds numbers for the various trials were $Re=$ 10, 100, and 200 and the corresponding Grashof numbers were $Gr=$ 2000,\ 20,000,\ and 40,000.
Note that $Re=200$, $Gr=$ 40,000, $Pr=0.72$ are the same parameter values reported in \cite{KIto_SSRavindran_1998b}.
The grids for the Reynolds number $10$ case had $N=21,184$ degrees of freedom while the grids for 
Reynolds number $100$ and $200$ cases had $N=58,184$ degrees of freedom.

\begin{table}[!hbt]
  \centering
  \begin{tabular}{ c | c c c c c c }
    Algorithm & Iters & Evals & $f(\bu)$ & $\|\nabla f(\bu)\|$ & $|c(\bu)|$ & $\|\nabla_{\bu} \mathcal{L}(\bu,\lambda) \|$ 	\\ \hline
    Uncon.\ NCG $Re=10$ & 10 & 11 & 1.43e-01 & 2.52e-06 & - & - \\ 
    Constrained $Re=10$ & 12 & 29 & 2.15e-01 & - & 5.27e-11 & 6.54e-07 \\ 
    ROM $Re=10$ & 5 & 6 & 2.15e-01 & - & 6.25e-11 & 4.39e-06 \\ 
    \hline
    Uncon.\ NCG $Re=100$ & 12 & 16 & 3.38e-01 & 1.86e-06 & - & - \\ 
    Constrained $Re=100$ & 41 & 105 & 4.09e-01 & - & 1.74e-09 & 2.25e-05 \\ 
    ROM $Re=100$ & 6 & 7 & 4.09e-01 & - & 4.26e-09 & 3.82e-06 \\ 
    \hline
    Uncon.\ NCG $Re=200$ & 22 & 57 & 4.93e-01 & 2.94e-05 & - & - \\ 
    Constrained $Re=200$ & 17 & 63 & 5.42e-01 & - & 3.04e-10 & 1.10e-05 \\ 
    ROM $Re=200$ & 6 & 7 & 5.42e-01 & - & 1.02e-09 & 3.55e-05 \\
  \end{tabular}
  \caption{Performance of the various optimization algorithms on the Boussinesq equations.
  Includes number of iterations, number of FOM function evaluations,
  final objective and first order necessary conditions.}
  \label{tb:exact-comp}
\end{table}

Table~\ref{tb:rom-results} compares the performance of the model based SQP Algorithm~\ref{alg:inexact} (ROM)
with the FOM only Algorithm~\ref{alg:exact} (FOM).
\begin{table}[tb]
  \centering
  \begin{tabular}{ c | c c c c c }
    Algorithm & Iters & FOM & ROM & (max) DOFs \\ \hline
    FOM $Re=10$ & 12 & 29 & - & 21,184 \\ 
    ROM $Re=10$ & 5 & 6 & 35 & 184 \\ 
    \hline
    FOM $Re=100$ & 41 & 105 & - & 58,184 \\ 
    ROM $Re=100$ & 6 & 7 & 70 & 192 \\ 
    \hline
    FOM $Re=200$ & 17 & 63 & - & 58,184 \\ 
    ROM $Re=200$ & 6 & 7 & 64 & 192 \\
  \end{tabular}
  \caption{Evaluation counts for exact SQP algorithm and SQP algorithm with models.
  Reports number of outer iterations, number of expensive FOM evaluations, number of ROM evaluations and number of degrees of freedom (DOFs) for state solve.
  The DOFs are constant for FOM, but the ROM size changes during the outer iterations.}
  \label{tb:rom-results}
\end{table}
Note that the runs reported are the same as those in Table~\ref{tb:exact-comp}.
Reported in the table is the number of iterations (Iters) that the algorithms required. 
Note that in the case of the ROM algorithm, this refers to the outer iterations and not the inner iterations required to approximately solve the subproblem~\eqref{eq:inexact-nlp}.
Also reported are the number of high fidelity FOM evaluations required, the number of low dimensional ROM solutions used,
and the number of degrees of freedom of the underlying state solve, in the ROM case this refers to the maximum size of the system.
For all test problems, the ROM algorithm outperformed the FOM only counterpart in terms of number of iterations and expensive evaluations required.
In the $Re=100$ case, the ROM algorithm was able to decrease the number of FOM evaluations by an order of magnitude.
Even when taking into account ROM evaluations (a significantly smaller system to solve) the total number of evaluations was comparable or favorable for the model based algorithm.
In all tests the ROMs were more than two orders of magnitude smaller than the corresponding FOM systems.

Figure~\ref{fig:outer-fonc-prog} below reports the per-iteration progress of the first order necessary conditions for each experiment, along with the penalty parameter values.
Shown in each of the left subfigures is the value of the norm of the constraint $\|c(\bu)\|_1$ versus the number of outer iterations
as well as $\|\nabla f(\bu) - c'(\bu)^T\lambda\|$ for both the ROM based algorithm and the FOM only algorithm.
It is important to note that for the ROM algorithm, each iteration reports outer iterations only.
In each outer iteration, several inner iterations (where the ROM only subproblem \eqref{eq:inexact-nlp} is being solved) may be taking place.
For the ROM plots, each dot represents one expensive FOM simulation so the direct comparison to the FOM only algorithm is appropriate.
For Reynolds number $100$, the FOM algorithm makes significant progress in the first few iterations then stalls out for several iterations.
In this case, this potential slowdown is avoided by the ROM algorithm since incremental progress is handled in the inner iterations (see Figure~\ref{fig:inner-fonc-prog}) where expensive evaluations are not used.
In the bottom figures, the value of the $\ell_1$-merit function is reported.
In all experiments, the penalty parameter $\rho_k$ remained fixed at one for both the ROM and FOM algorithms.
This is reflected in the right plots by the lack of notable decrease of the $\ell_1$-merit function in later iterations.
Since the penalty parameter is not increased, once the constraint violation becomes small, the value of the penalty parameter is dominated by the objective value.
\begin{figure}[!htb]
	  \includegraphics[width=0.32\textwidth]{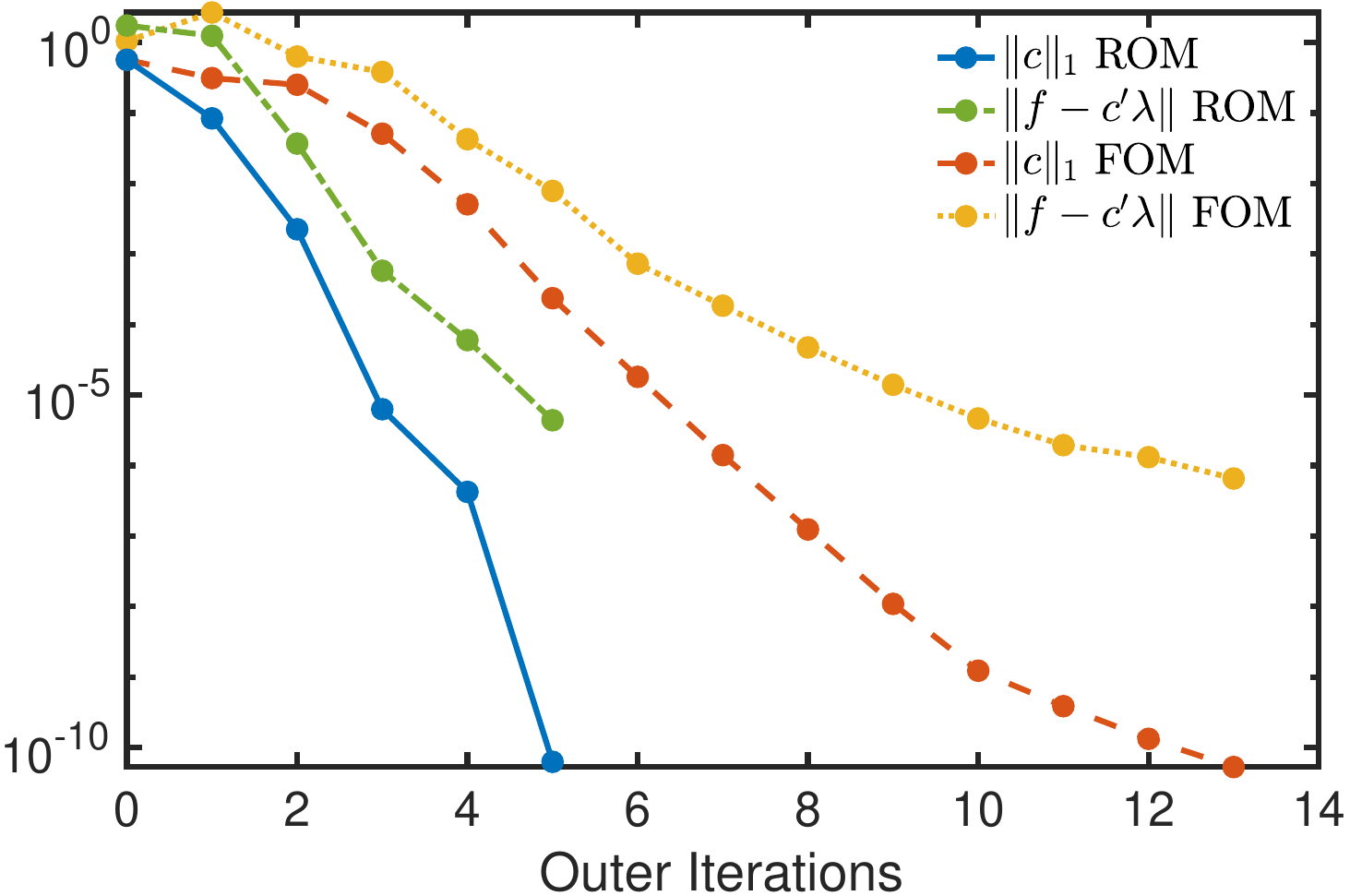}\hfill
	  \includegraphics[width=0.32\textwidth]{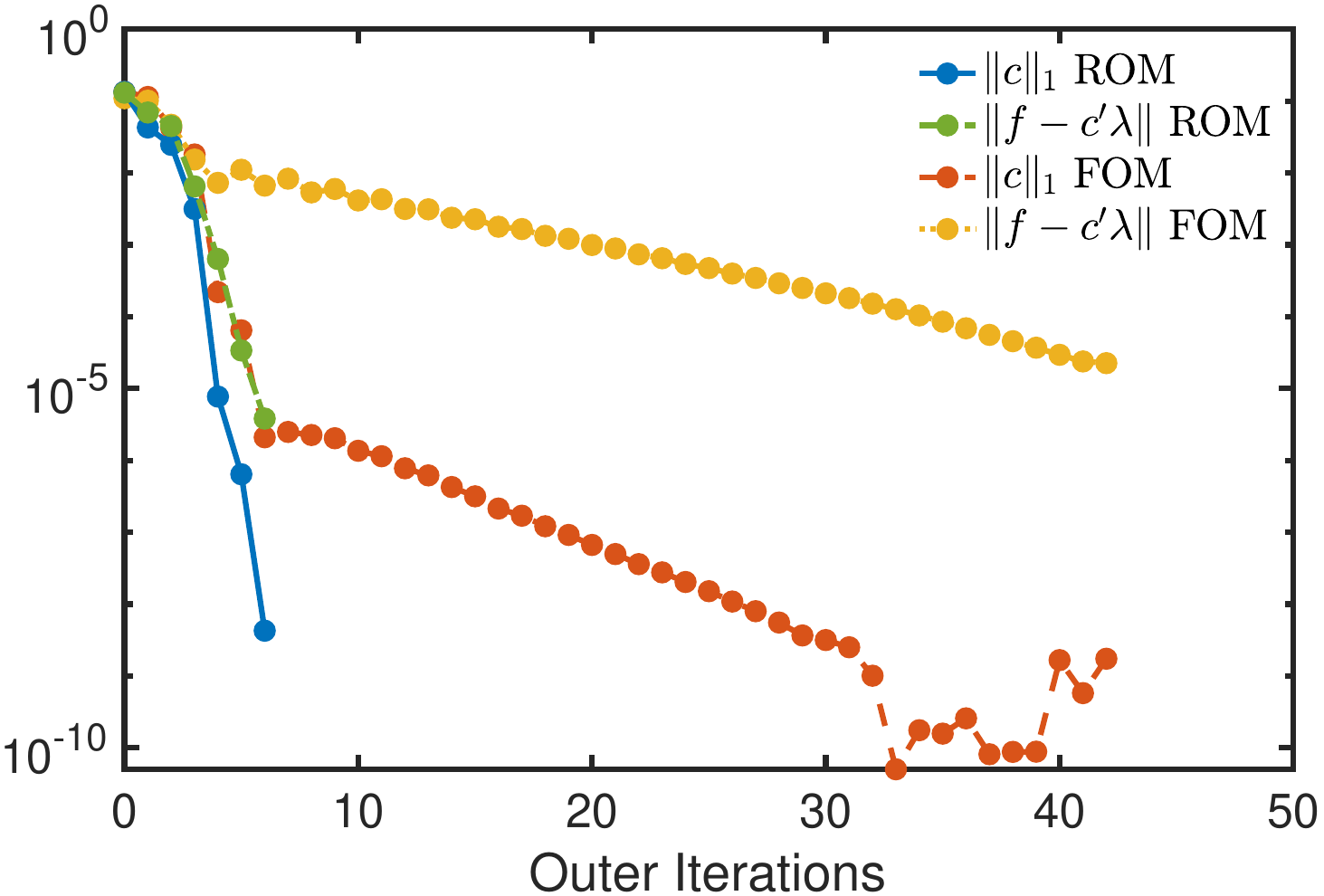} \hfill
	  \includegraphics[width=0.32\textwidth]{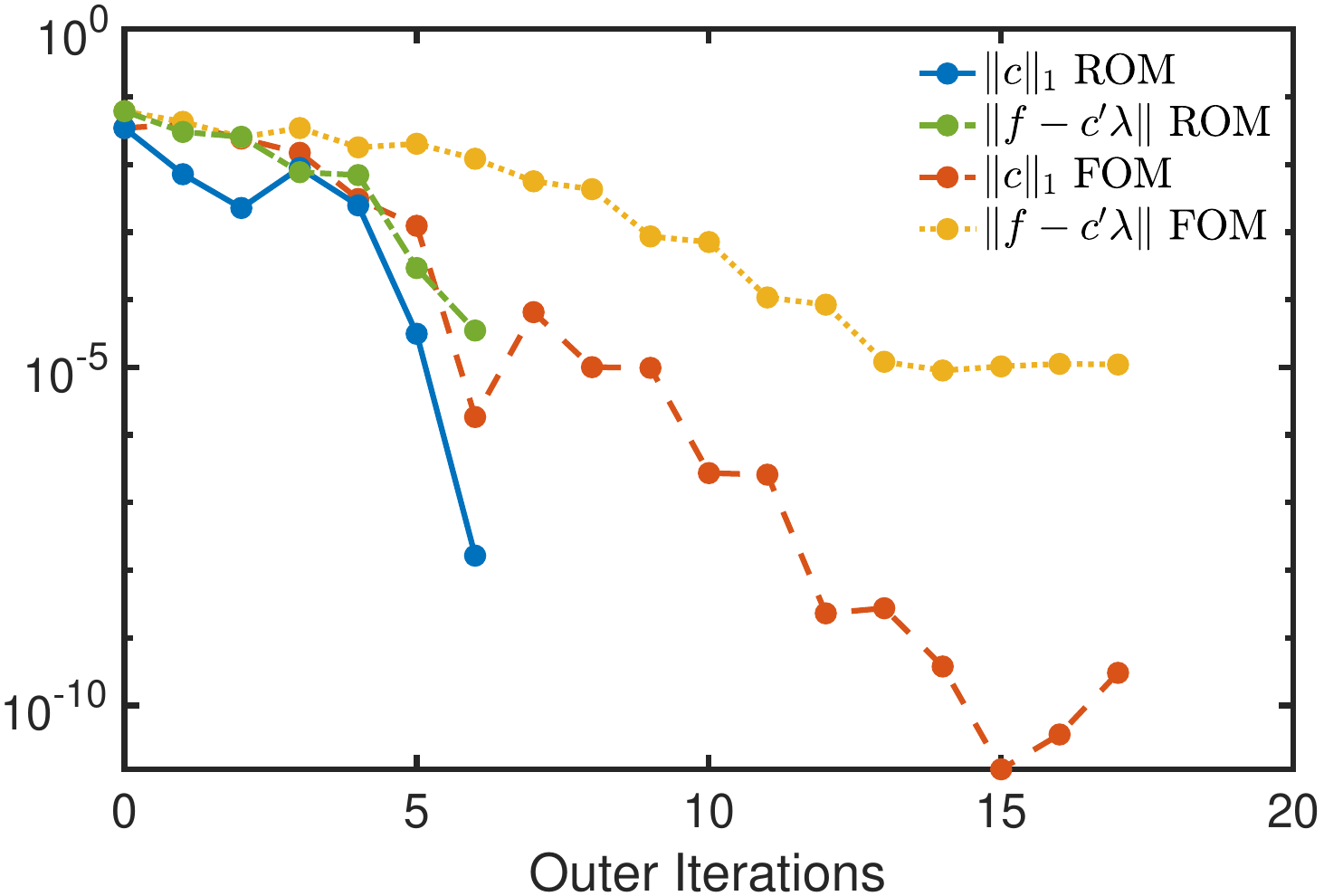} 
	  
	  \includegraphics[width=0.32\textwidth]{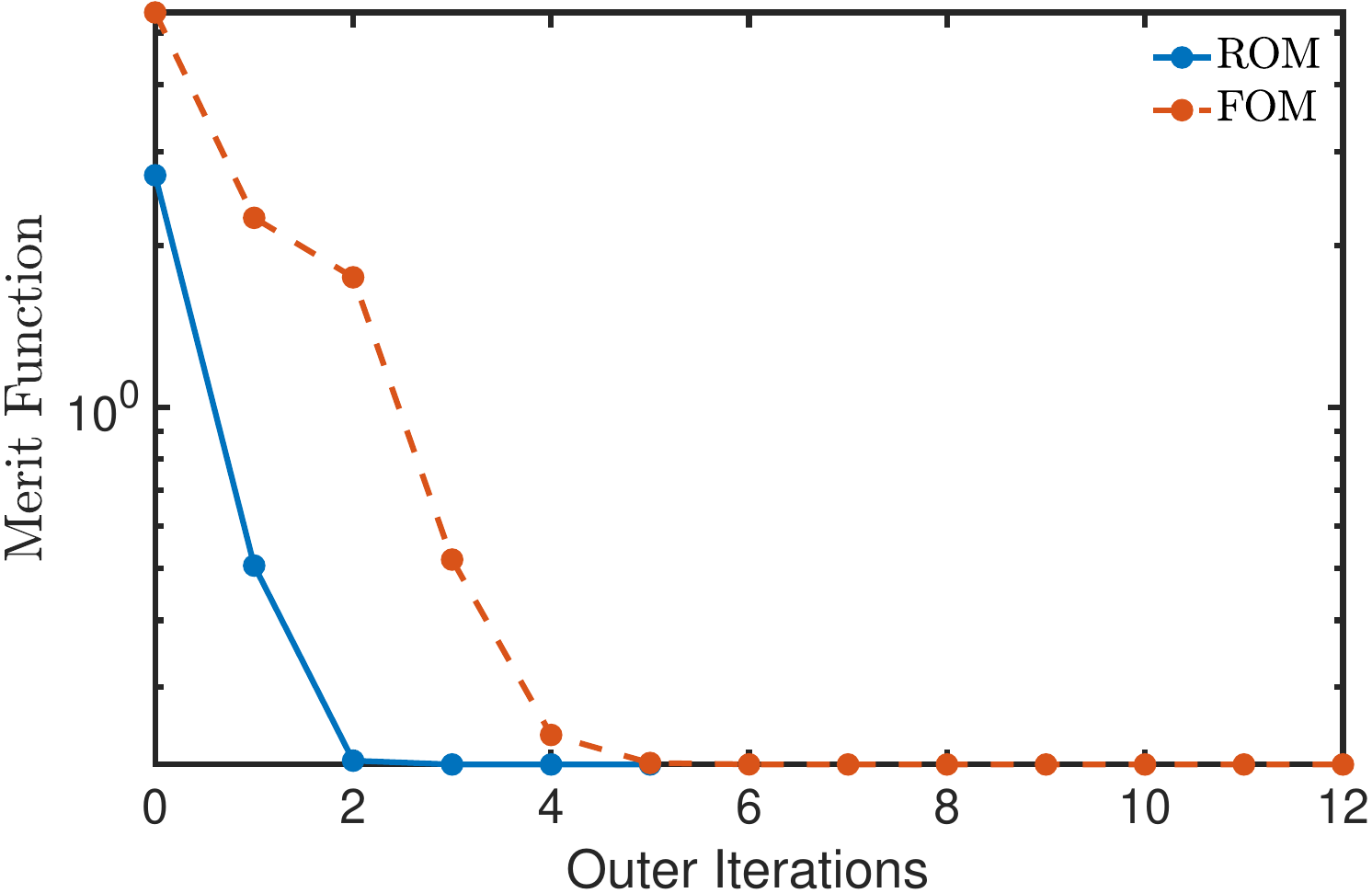} \hfill
	  \includegraphics[width=0.32\textwidth]{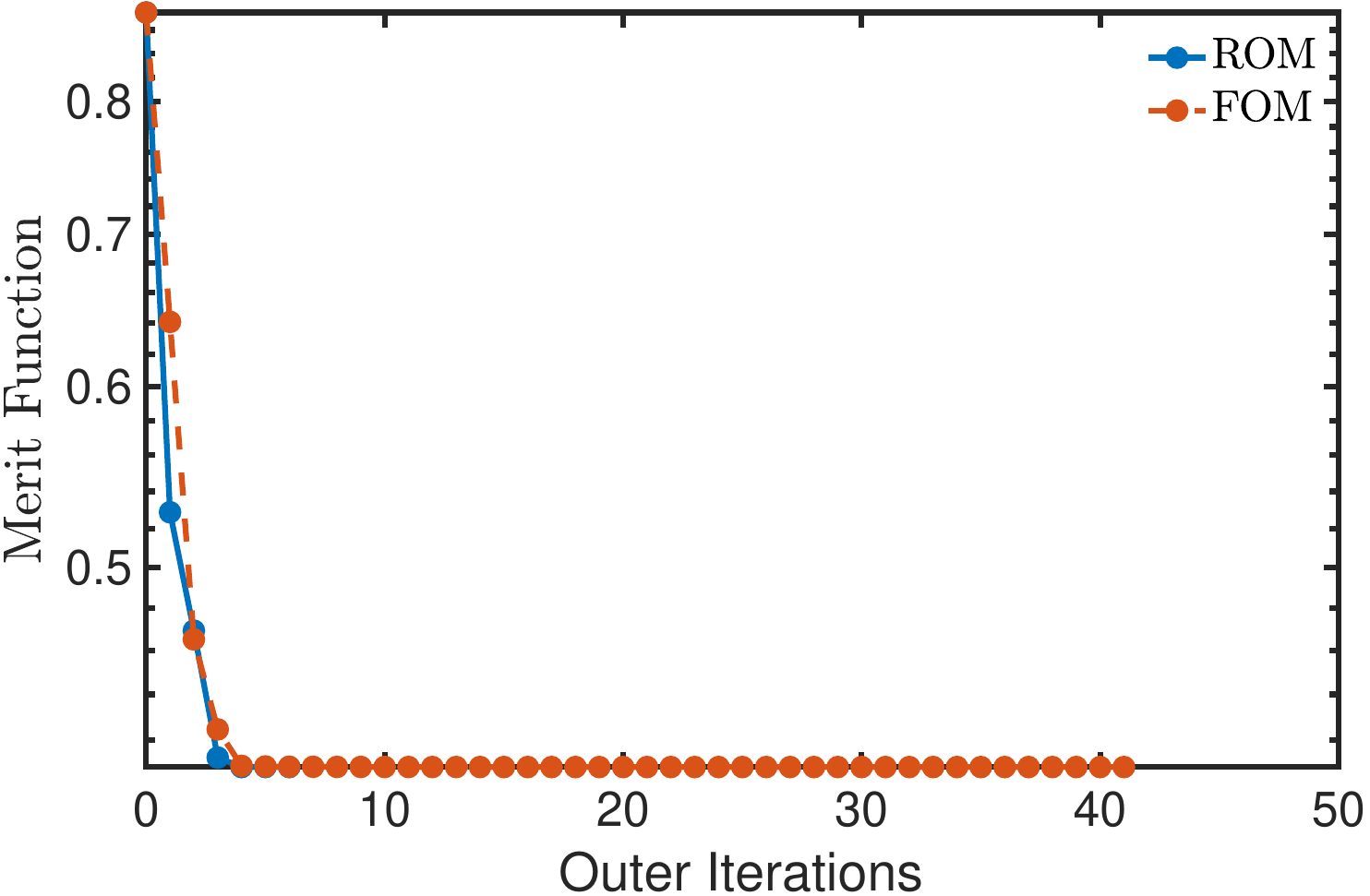}  \hfill
	  \includegraphics[width=0.32\textwidth]{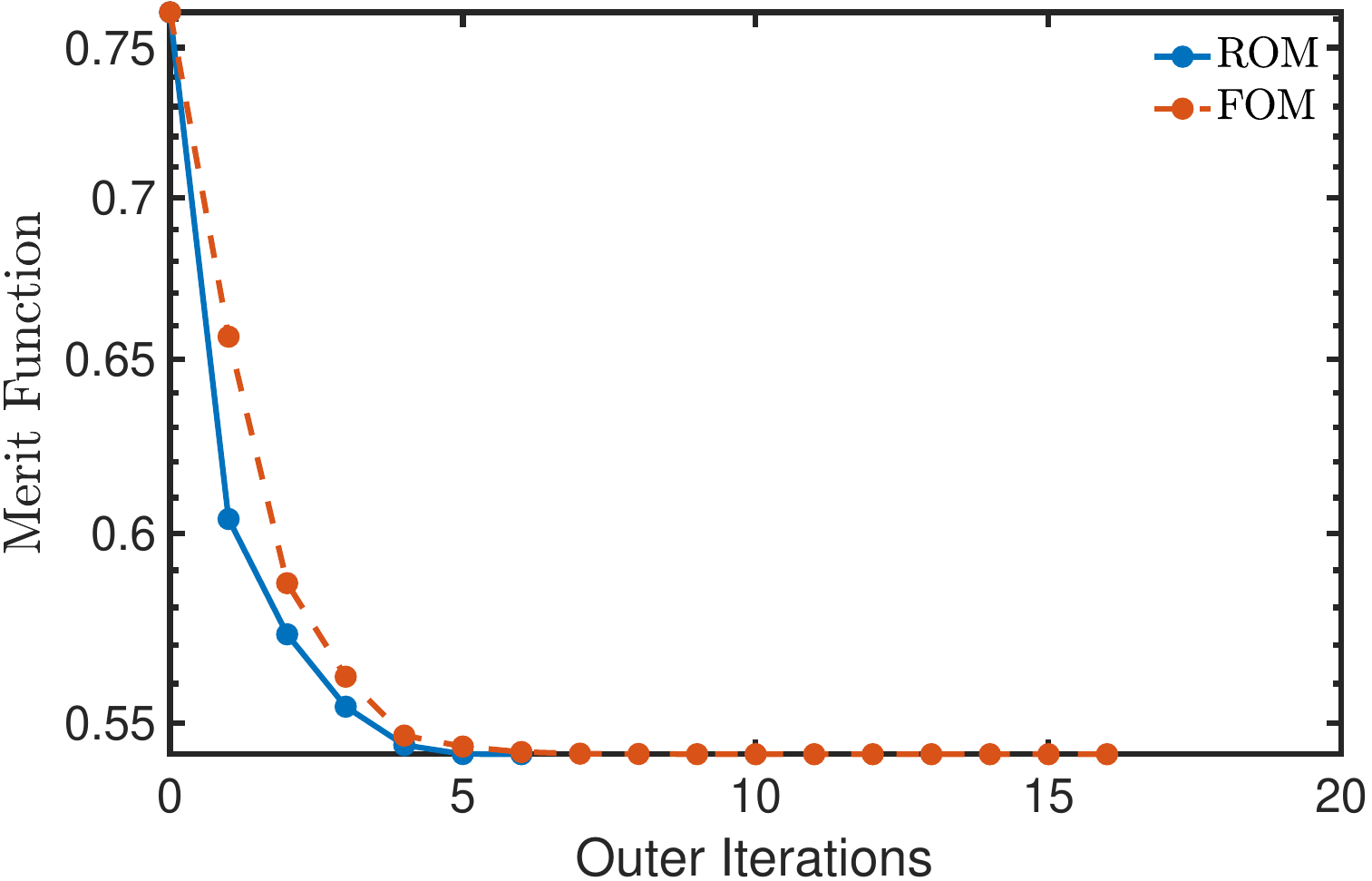} 

  \caption{Iteration history for the first order necessary conditions and penalty parameter.
    Compares progress for ROM algorithm and FOM only algorithm, ROM reports only outer iterations.
    Each column corresponds to a different experiment, left is $Re=10$, middle is $Re=100$, rught is $Re=200$.
    Plots at the tops show the first order necessary conditions, plots at the bottom show the $\ell_1$-merit function values.}
	\label{fig:outer-fonc-prog}
\end{figure}

Figure~\ref{fig:inner-fonc-prog} compares the progress of the ROM algorithms for the different values of the Reynolds number.
The left plot reports $\|h_k(\bu)\|_1$ versus the total number of iterations, the right plot reports $\|\nabla m_k(\bu) - h_k'(\bu)^T\lambda\|$.
The large dots are the outer iterations, while the smaller dots are inner iterations.
The value of the first-order necessary conditions increases in outer iterations because the FOM is solved and a new ROM is constructed.
All values reported are for the ROM models, however, the outer iteration values match the true objective because the FOM is solved at those values.
One of the main advantages of the ROM algorithm is highlighted by the results in Figure~\ref{fig:inner-fonc-prog};
in later iterations where the FOM only algorithm tends to stall out, the ROM algorithm continues to make progress during outer iterations.
\begin{figure}[t]
	\includegraphics[width=0.42\textwidth]{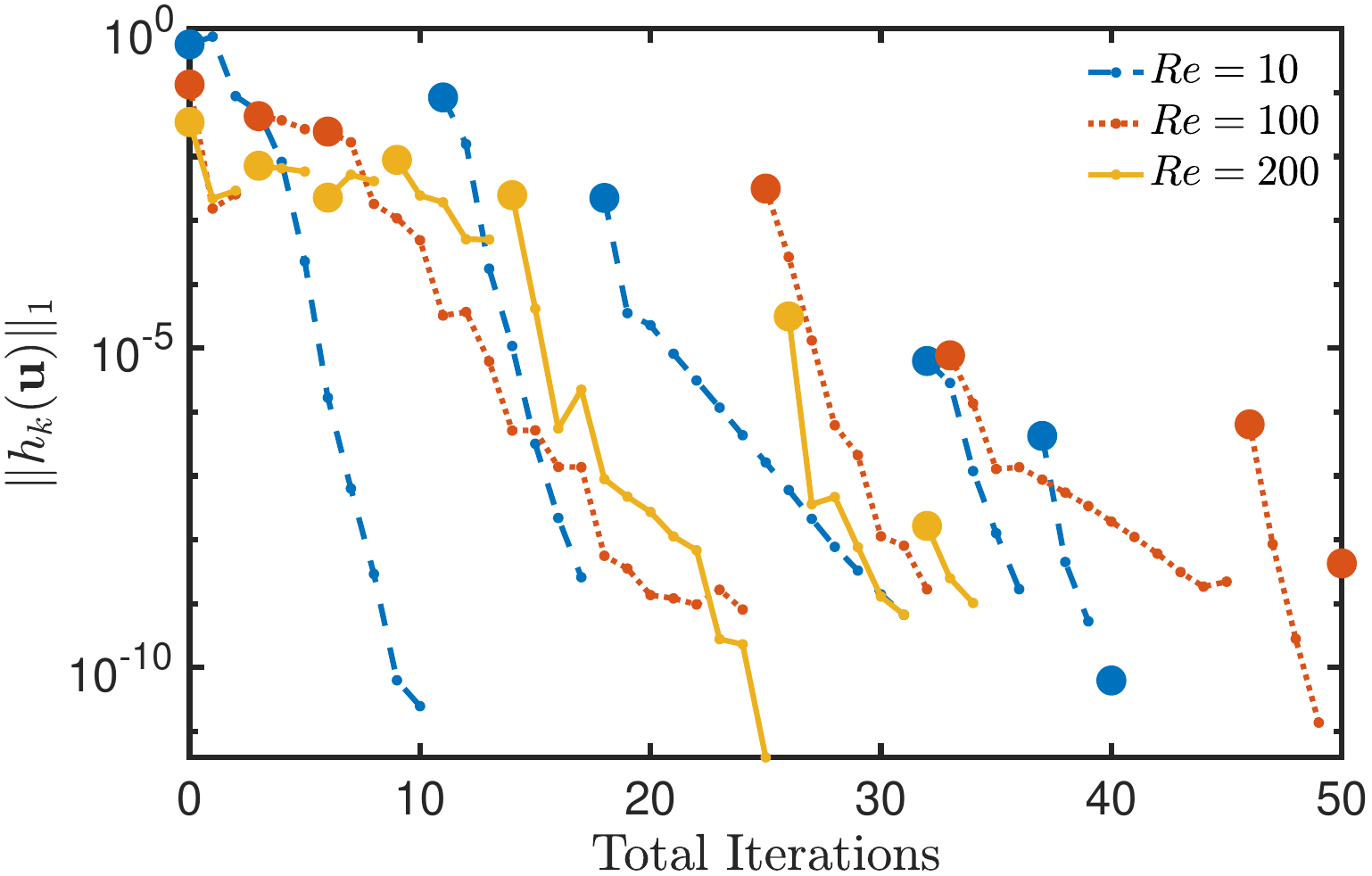} \hfill
	\includegraphics[width=0.42\textwidth]{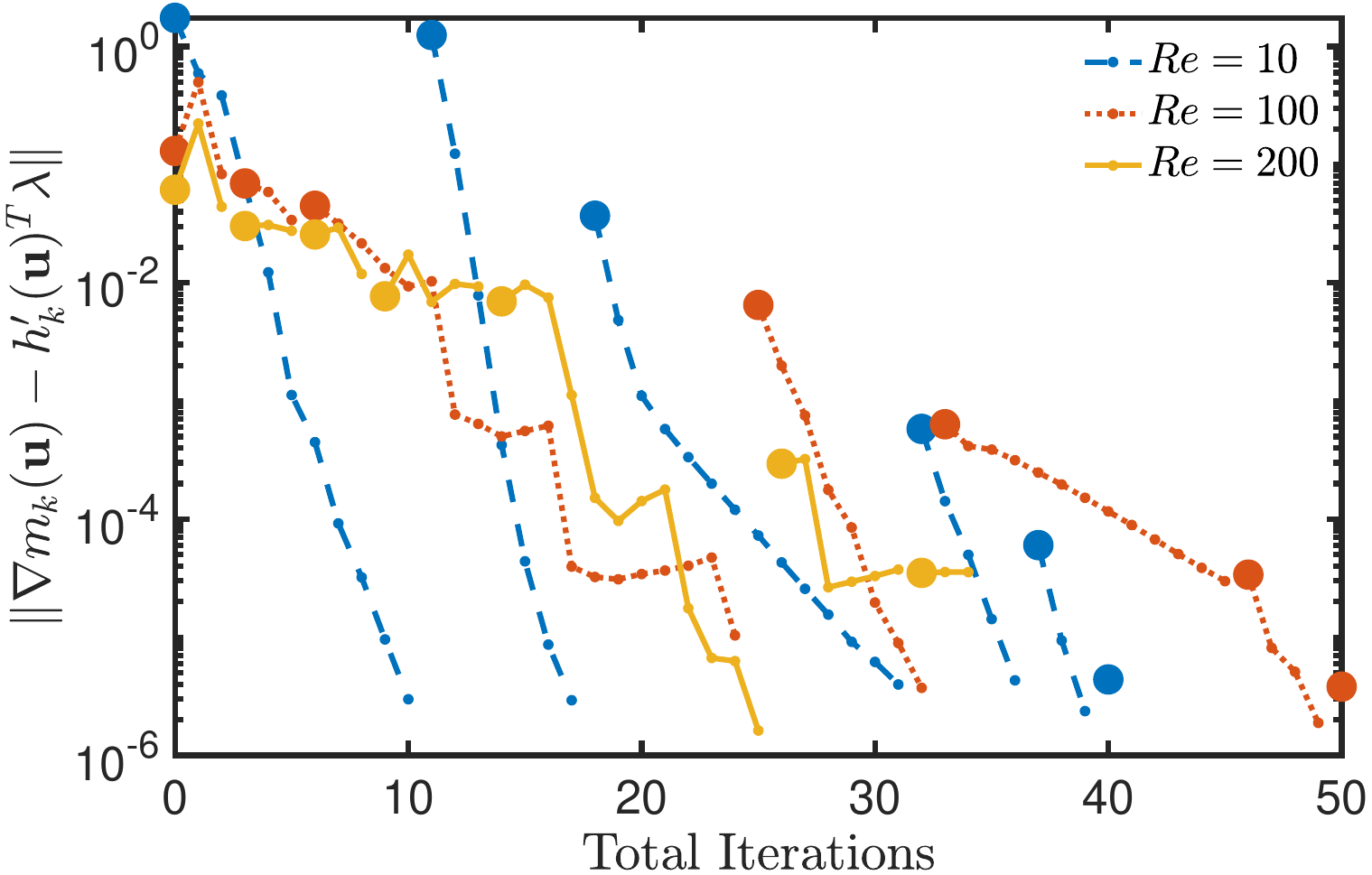}
	
  \caption{Iteration history for the first order necessary conditions.
    Large dots represent outer iterations, small dots inner iterations.
    Left plot shows the constraint $\|h(\bu)\|_1$, right plot $\|m_k(\bu) - h'(\bu)^T\lambda_{k+1}\|$.}
	\label{fig:inner-fonc-prog}
\end{figure}

The remaining figures show, in order, the temperature boundary control \eqref{eq:ecmod:BS:boussi-eq-cntrl} for the unconstrained and constrained problem.
The zero control velocity profile in the region $D$ behind the step where vorticity is being controlled,
and the optimal velocity profiles for the unconstrained and constrained problems.
The zero control temperature profile with the optimal profiles for the unconstrained and constrained problems.
Because of page limitations, we only show the plots for $Re=200$. For additional results and plots see \cite{DSGrundvig_2025a}.


\begin{figure}[!hbt]
          \includegraphics[width=0.49\textwidth]{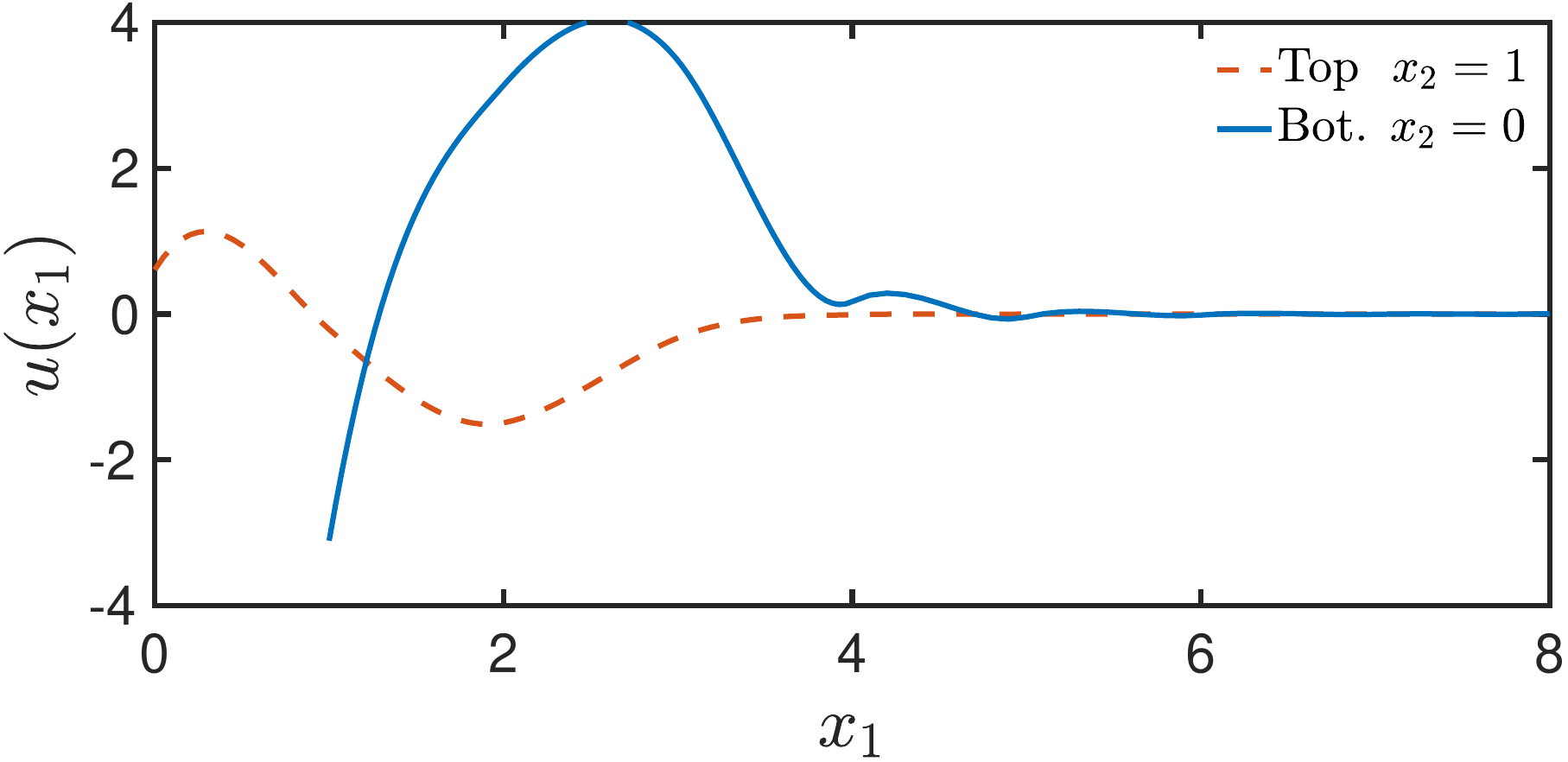} \hfill
	  \includegraphics[width=0.49\textwidth]{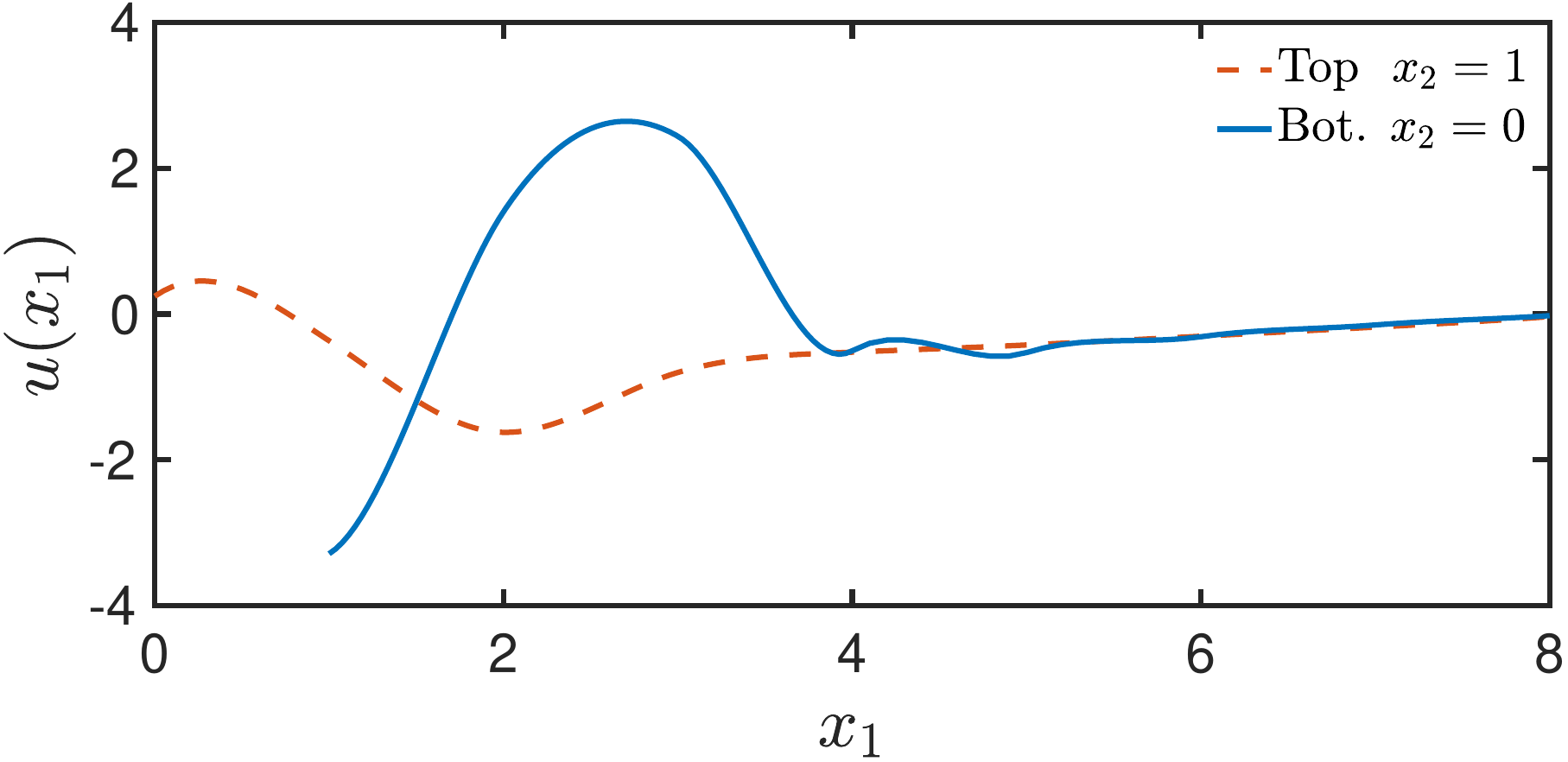}

  \caption{Optimal control for $Re=200$, $Gr=$ 40,000, $Pr=0.72$ with 58,184 degrees of freedom.
  Left is unconstrained problem, right is constrained problem.}
	\label{fig:opt-control-unc-Re200}
\end{figure}

%

\begin{figure}[!hbt]
         \includegraphics[width=0.32\textwidth]{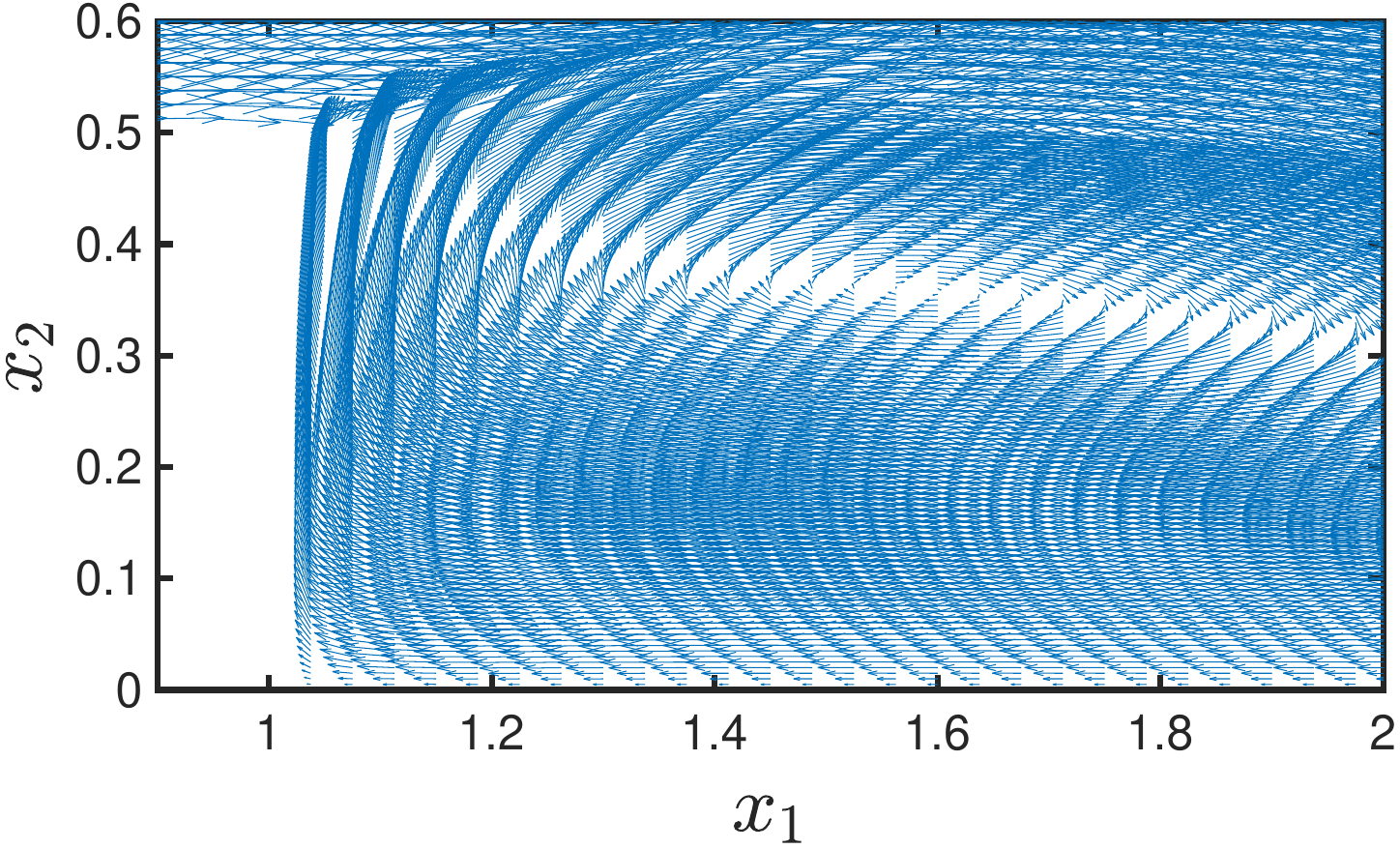} \hfill
 	 \includegraphics[width=0.32\textwidth]{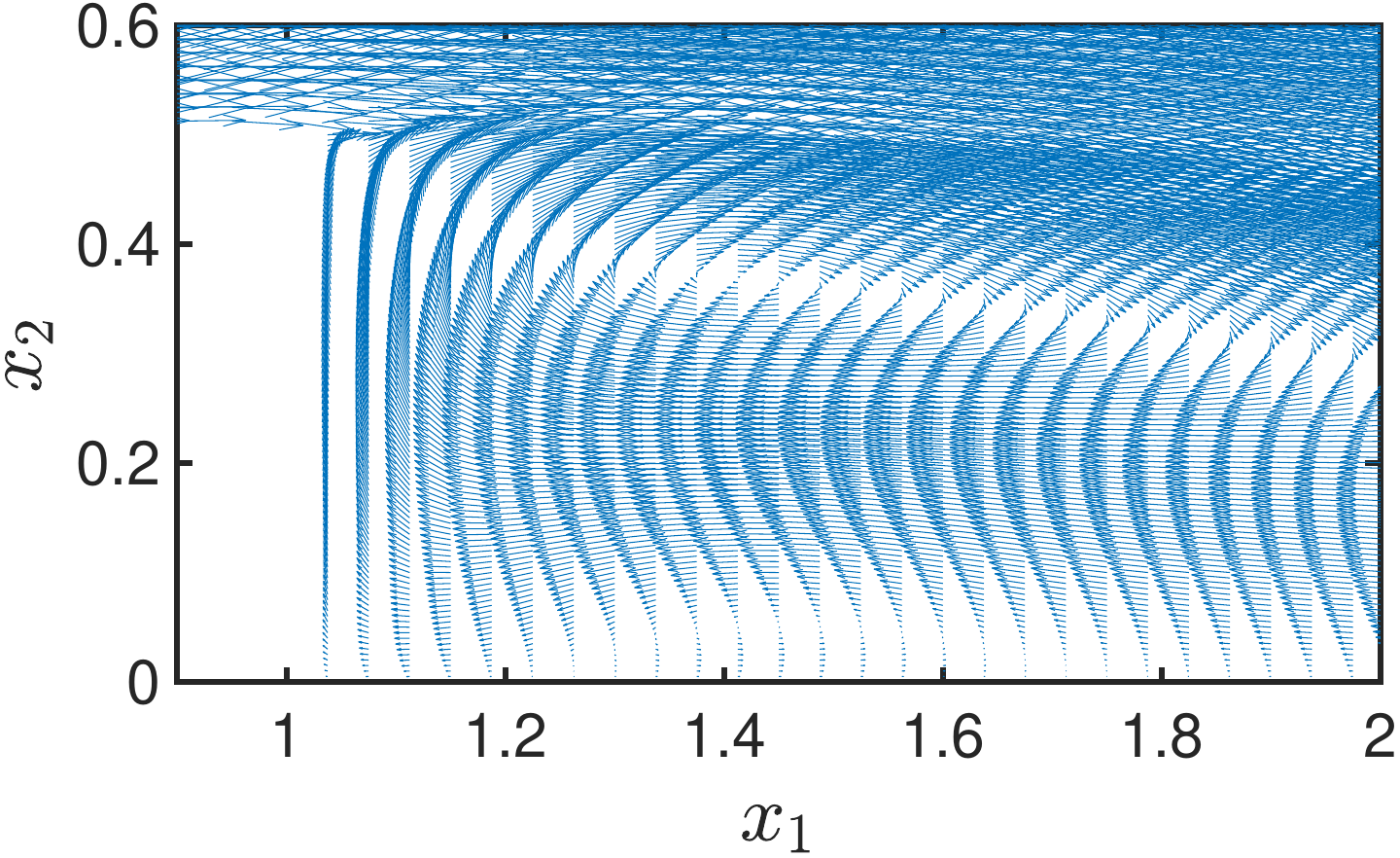}  \hfill
         \includegraphics[width=0.32\textwidth]{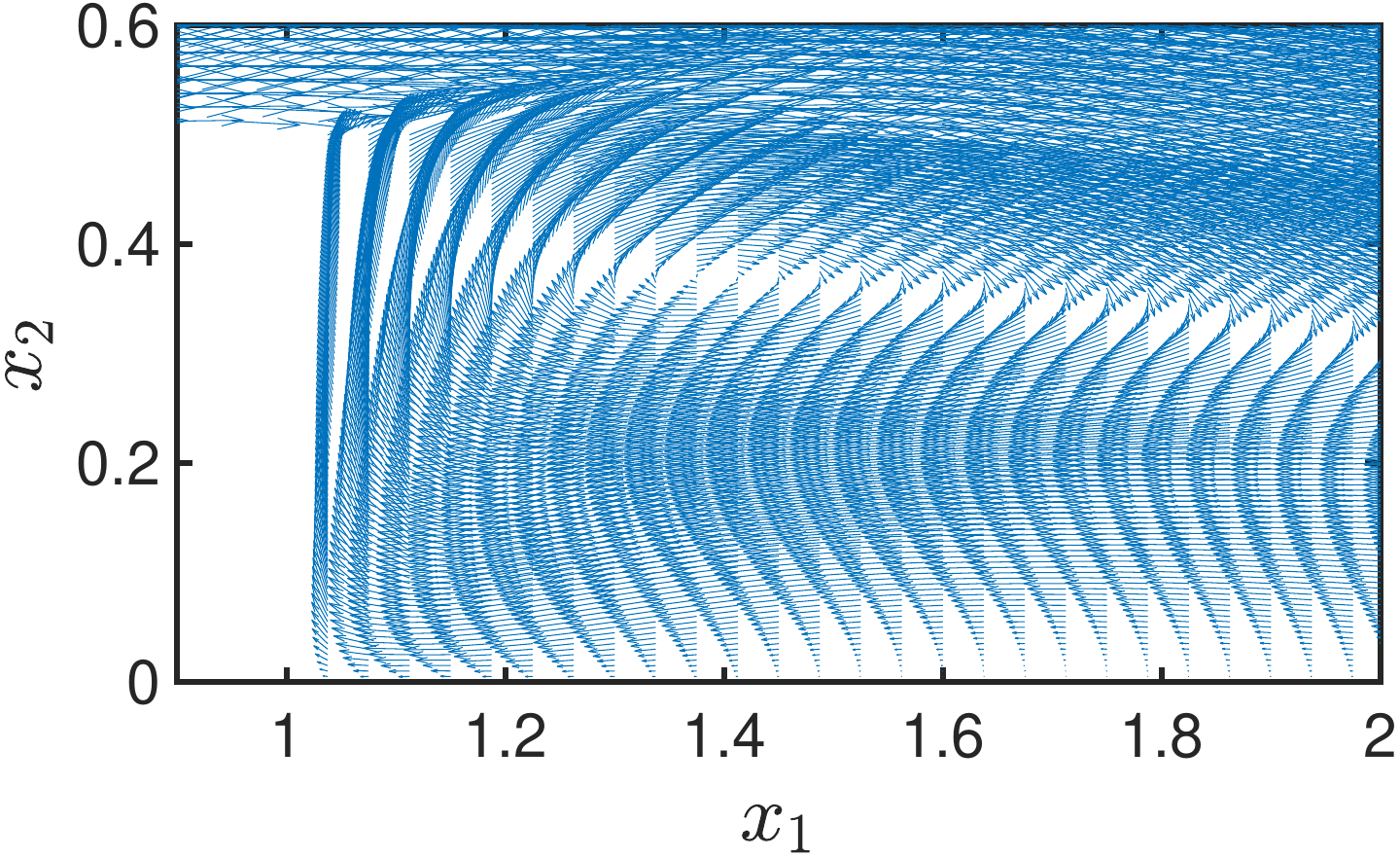}

  \vspace*{-2ex}
  \caption{Uncontrolled velocity (left), optimal velocity for unconstrained (middle), and optimal velocity for constrained (right) problems,
  $Re=200$, $Gr=$ 40,000, $Pr=0.72$ with 58,184 degrees of freedom.}
	\label{fig:channel-Re200}
\end{figure}

\begin{figure}[!hbt]
	\centering
  \includegraphics[width=0.9\textwidth]{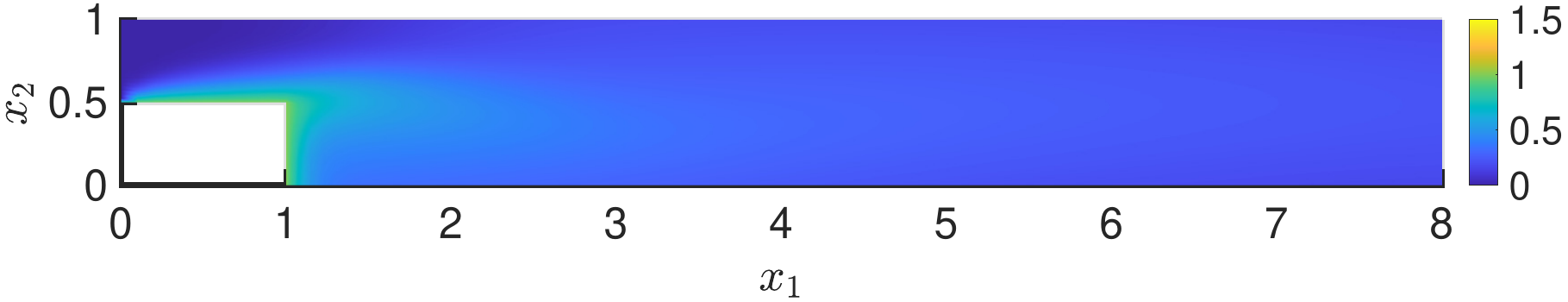}
	\includegraphics[width=0.9\textwidth]{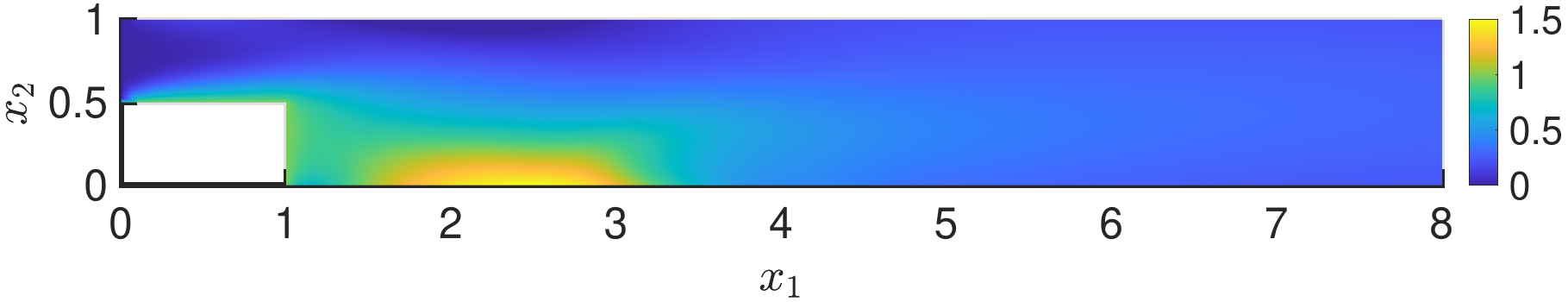}
	\includegraphics[width=0.9\textwidth]{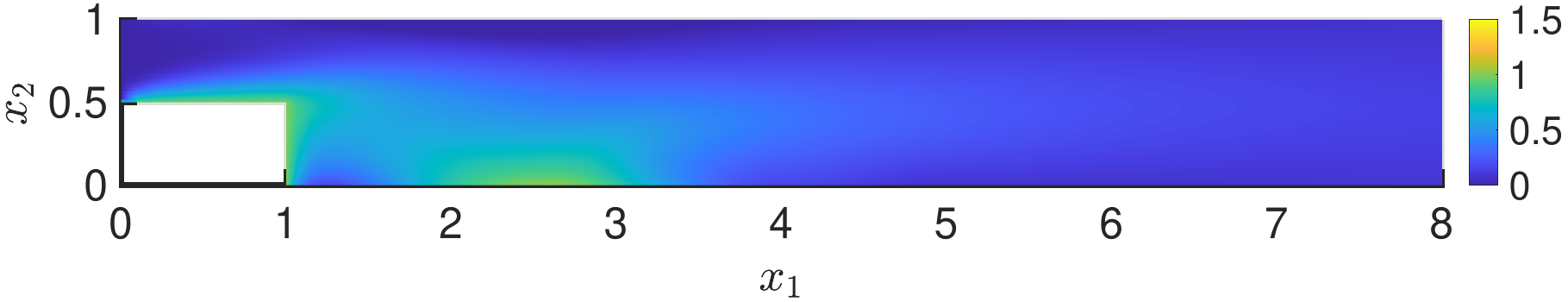}
	
  \vspace*{-1ex}
  \caption{Temperature plot comparing uncontrolled (top), optimal unconstrained (middle) and optimal constrained (bottom) problems,
  $Re=200$, $Gr=$ 40,000, $Pr=0.72$ with 58,184 degrees of freedom.}
	\label{fig:temperature-Re200}
\end{figure}

\pagebreak
\section{Conclusions}

We have introduced a generalized line-search SQP algorithm with $\ell_1$-merit function
that uses objective function and constraint function models with tunable accuracy for the solution of smooth equality-constrained optimization problems.
The algorithm was motivated by an algorithm developed in \cite{DSGrundvig_MHeinkenschloss_2025a} for unconstrained problems.
Our algorithm leverages efficient approximate models with error bounds and tunable accuracy to reduce the number of expensive objective and constraint function evaluations.
We presented a general convergence result for our generalized line-search SQP algorithm and provided the details of an implementable algorithm.
Specifically, we proved that our algorithm has the same first-order global convergence properties as the traditional  line-search SQP algorithm with $\ell_1$-merit function,
provided the objective and constraint function models meet the required accuracy requirements. Our algorithm only uses these models and 
corresponding error functions, but never directly accesses the original objective function. The accuracy requirement of the model are adjusted 
to the progress of the optimization algorithm.
In addition, we proved that for a large class of commonly used model functions based on ROMs, the constructed models meet the
accuracy requirements specified by our algorithm.

Our algorithm was applied to a boundary control problem governed by the Boussinesq PDE, and models were computed using ROMs.
We demonstrated that with appropriately chosen models, the algorithm demonstrates significant reductions both in the number of outer iterations
and in the number of expensive full-order PDE solves required to reach an optimal solution, when compared to optimization strategies 
using solely exact objective and constraint function information.

There are several avenues for future research.
Currently, no hyperreduction is used in the Boussinesq PDE ROMs. However, it seems possible to adapt the
techniques from \cite{TWen_MJZahr_2025a} to introduce hyperreduction in a systematic way to arrive at truly
efficient algorithms. This would combine the hyperreduction of \cite{TWen_MJZahr_2025a} with the comprehensive
convergence analysis of our optimization algorithm.

As reported in \cite{DSGrundvig_MHeinkenschloss_2025a} for unconstrained problems, the construction of the ROM
has a big impact on the overall performance. We include sensitivity information to construct ROMs, which has shown
to be beneficial, e.g., in \cite{DSGrundvig_MHeinkenschloss_2025a}, \cite{TWen_MJZahr_2023a},
However, when sensitivity information is added, the ROM size depends on the number of optimization variables.
The shape optimization examples in  \cite{DSGrundvig_MHeinkenschloss_2025a}, \cite{TWen_MJZahr_2023a} 
have a relatively small number of shape parameters and ROM sizes stay small, whereas in our example, the number of 
discrete controls was relatively large ($n_u=34$), resulting in larger ROMs. Exploration of other ROM constructions will be
beneficial. 

Finally, numerical exploration using different asymptotic error estimates on the overall performance are useful.
For unconstrained problems, results in this direction are contained in \cite[Ch.~3]{DSGrundvig_2025a}.


\bibliography{references}
\bibliographystyle{plainurl}

\end{document}